\documentclass[12pt,a4paper]{amsart}

\usepackage{amssymb,latexsym,graphicx}
\usepackage{color}
\usepackage{colortbl}
\usepackage{subcaption, float}
\usepackage[cp1252]{inputenc} 
\usepackage[T1]{fontenc}
\usepackage{lmodern}
\usepackage[english]{babel}
\usepackage[section]{placeins}
\usepackage[pagebackref=true]{hyperref}
\hypersetup{pdftex,colorlinks=true,linkcolor=blue,citecolor=blue,urlcolor=blue}

\usepackage{currfile}

\usepackage[margin=3cm]{geometry}
\usepackage[normalem]{ulem}

\newtheorem{theorem}{Theorem}[section]
\newtheorem{lemma}[theorem]{Lemma}
\newtheorem{proposition}[theorem]{Proposition}
\newtheorem{definition}[theorem]{Definition}

\newtheorem{corollary}[theorem]{Corollary}

\newtheorem{remark}[theorem]{Remark}
\newtheorem{remarks}[theorem]{Remarks}
\newtheorem{notation}[theorem]{Notation}
\newtheorem{property}[theorem]{Property}
\newtheorem{properties}[theorem]{Properties}

\numberwithin{equation}{section}
\numberwithin{figure}{section}
\numberwithin{table}{section}

\definecolor{purple}{RGB}{127,0,255}
\definecolor{gray}{RGB}{128,128,128}

\definecolor{lgray}{gray}{0.90}

\newcommand{\noid}{\noindent $\diamond$~}

\newcommand{\C}{\mathbb{C}}

\newcommand{\Nb}{\mathbb{N^{\bullet}}}

\newcommand{\R}{\mathbb{R}}
\newcommand{\bS}{\mathbb{S}}

\newcommand{\Z}{\mathbb{Z}}

\newcommand{\cA}{\mathcal{A}}
\newcommand{\cC}{\mathcal{C}}
\newcommand{\cD}{\mathcal{D}}
\newcommand{\cE}{\mathcal{E}}

\newcommand{\cR}{\mathcal{R}}

\newcommand{\cZ}{\mathcal{Z}}

\newcommand{\lambdah}{\hat{\lambda}}

\newcommand{\sm}{\!\setminus\!}

\DeclareMathOperator{\id}{Id}
\DeclareMathOperator{\arccot}{arccot}
\newcommand{\set}[1]{\left\lbrace #1 \right\rbrace}
\newcommand{\pib}{\frac{\pi}{2}}
\newcommand{\pic}{\frac{\pi}{3}}
\newcommand{\pit}{\frac{2\pi}{3}}
\newcommand{\pid}{\frac{\pi}{4}}
\newcommand{\pif}{\frac{\pi}{6}}

\newcommand{\vo}[1]{\hphantom{#1}}
\newcommand{\pf}{\emph{Proof.~}}

\title[Courant-sharp property]{Courant-sharp property for Dirichlet eigenfunctions  on the M\"{o}bius strip}

\date{\today}

\author[P. B\'{e}rard]{Pierre B\'erard}
\author[B. Helffer]{Bernard Helffer}
\author[R. Kiwan]{Rola Kiwan}

\address{PB: Universit\'{e} Grenoble Alpes and CNRS\\
Institut Fourier, CS 40700\\ 38058 Grenoble cedex 9, France.}
\email{pierrehberard@gmail.com}

\address{BH: Laboratoire Jean Leray, Universit\'{e} de Nantes and CNRS\\
F44322 Nantes Cedex, France and LMO (Universit\'e Paris-Sud).}
\email{Bernard.Helffer@univ-nantes.fr}

\address{RK: American University in Dubai, P.O.Box 28282, Dubai, United Arab Emirates.}
\email{rkiwan@aud.edu}



\date{\today ~~To appear in \emph{Portugaliae Mathematica}~(\currfilename)}

\begin{document}

\begin{abstract}{The question of determining for which eigenvalues there exists an eigenfunction which has the same number of nodal domains as the label of the associated eigenvalue (Courant-sharp property) was motivated by the analysis of minimal spectral partitions. In previous works, many examples have been analyzed corresponding to squares, rectangles, disks, triangles, tori, \ldots . A natural toy model for further investigations is the M\"{o}bius strip, a non-orientable surface with Euler characteristic $0$, and particularly the ``square'' M\"{o}bius strip whose eigenvalues have higher multiplicities. In this case, we prove that the only Courant-sharp Dirichlet eigenvalues are the first and the second, and we exhibit peculiar nodal patterns.}
\end{abstract}

\maketitle

\section{Introduction}\label{S-int}

The question we are interested in was initially suggested by \'{E}tienne Ghys in March 2016, during a conference in Abu Dhabi.\medskip

We start with the standard strip,
$$
S_\infty:= (0,\pi) \times (-\infty,+\infty)\,,
$$
and we look at the standard Laplacian $-\Delta$ with Dirichlet condition at $x=0$ and  $ x=\pi$, and we add the conditions,
\begin{equation}\label{cond1}
 u(x, y + \pi a) = u (\pi -x,y)\,,\, u\in H^2_{loc}(\overline{S_\infty}) \,,
\end{equation}
where $a$ is a positive parameter. Equivalently, we look at the Laplacian on the flat M\"{o}bius strip $M_a$, with Dirichlet boundary condition (see Section~\ref{S-moeb}).

According to Courant's nodal domain theorem \cite[Chap. {VI}.6]{CH1953}, an eigenfunction associated with the $n$th Dirichlet eigenvalue of $M_a$ has at most $n$ nodal domains. The eigenvalue $\lambda_n$ is called \emph{Courant-sharp} if there exists an associated eigenfunction with exactly $n$ nodal domains. As usual, we list the eigenvalues in nondecreasing order, multiplicities accounted for, starting from the label $1$.

\begin{remark}\label{R-int-2} There are obvious restrictions for an eigenvalue to be Courant-sharp. Indeed, let $k \ge 2$ be an integer. If the eigenvalue $\lambda_k$ satisfies $\lambda_{k-1} < \lambda_k = \cdots = \lambda_{k+\ell-1} < \lambda_{k+\ell}$ for some integer $\ell \ge 2$, then the eigenvalues $\lambda_{k+1},\ldots,\lambda_{k+\ell-1}$ cannot be Courant-sharp.
\end{remark}%

Our aim is to prove the following theorem.

\begin{theorem}\label{main}
When $a=1$, the only Courant-sharp eigenvalues of the Dirichlet Laplacian on the M\"{o}bius strip $M_1$ are the first and second eigenvalues.
\end{theorem}

Equivalently, for the Laplacian with Dirichlet condition on the boundary of $S_{\infty}$ and the periodicity condition \eqref{cond1}, with $a=1$, the only Courant-sharp eigenvalues are the first and second ones.\medskip

Here is a sketch of the proof. Since the M\"{o}bius strip is a surface with boundary (actually a quotient of a cylinder), we can apply \cite[p.~524]{BM} to extend Pleijel's theorem \cite{Pl}, and conclude that there exists an explicit number $\gamma(2) < 1$ such that
$$
\limsup_{k \longrightarrow +  \infty}  \frac{\nu(k)}{k}\leq \gamma(2) <1\,,
$$
where $\nu(k)$ denotes the maximal number of nodal domains of an eigenfunction associated with the eigenvalue $\lambda_k$. This inequality implies that there are finitely many Courant-sharp eigenvalues only. Using Weyl's law with a controlled remainder term, and an adapted Faber-Krahn inequality, we give an upper bound and a condition to be satisfied by Courant-sharp eigenvalues. Together with Remark~\ref{R-int-2}, these conditions limit the possible Courant-sharp eigenvalues of $M_1$ to the set $\set{\lambda_1,\lambda_2,\lambda_6,\lambda_7}$. Courant's theorem implies that $\lambda_1$ and $\lambda_2$ are Courant-sharp. To conclude the proof, we determine the possible nodal patterns of the eigenfunctions associated with $\lambda_6$ and $\lambda_7$.\medskip

Because of its symmetries, and higher eigenvalue multiplicities, the case of the ``square'' M\"{o}bius strip $M_1$ ($a=1$) seems to be the most interesting. Other investigations corresponding to irrational or small $a$'s could also be performed as  in \cite{HHOT}, leading to partial answers to the Courant-sharp question. In this paper, we shall only consider the case $a=1$. \medskip

Looking for the Courant-sharp eigenvalues of the flat M\"{o}bius strip comes naturally in view of the known results for the square and for the flat tori. It is also natural to consider cylinders and Klein bottles. The following table displays some of the known results.

\begin{table}[!htb]
\centering
\resizebox{0.98\textwidth}{!}{%
\begin{tabular}[c]{|c|c|c|c|}%
\hline &&&\\[2pt]
Case & Boundry Cond. & Courant-sharp eigenvalues & References\\[7pt]
\hline
Square & Dirichlet & $\set{\lambda_1,\lambda_2,\lambda_4}$ & \cite{Pl,BH}\\[5pt]
\hline
Square & Neumann & $\set{\lambda_1,\lambda_2,\lambda_4,\lambda_5,\lambda_9}$ & \cite{HePeSu2015}\\[5pt]
\hline
Square & Robin & $\set{\lambda_1,\lambda_2,\lambda_4}$ & \cite{GiHe2019} ($h$ large, Remarks~\ref{R-T})\\[5pt]
\hline
Square & Robin & $\set{\lambda_1,\lambda_2,\lambda_4,\lambda_9}$ & \cite{GiHe2020} ($h > 0$ small, Remarks~\ref{R-T})\\[5pt]
\hline
Torus & -- & $\set{\lambda_1,\lambda_2}$ & \cite{BeHe2016,Le} (resp. for the flat equilateral or square tori)\\[5pt]
\hline
Triangle & Dirichlet & $\set{\lambda_1,\lambda_2,\lambda_4}$ & \cite{BeHe2016} (equilateral triangle)\\[5pt]
\hline
Disk & Dirichlet/Neumann & $\set{\lambda_1,\lambda_2,\lambda_4}$ & \cite{HHOT,HePeSu2016}\\[5pt]
\hline
\hline
M\"{o}bius strip & Dirichlet & $\set{\lambda_1,\lambda_2}$ & The present paper\\[5pt]
\hline
Cylinder & Dirichlet & $\set{\lambda_1,\lambda_2}$ & \cite{BeHeKi2020} (See Remarks~\ref{R-T})\\[5pt]
\hline
Klein bottle& -- & $\set{\lambda_1,\lambda_2}$ & \cite{BeHeKi2020} (See Remarks~\ref{R-T}) \\[5pt]
\hline
\end{tabular}}
\medskip
\caption{}\label{T-int}
\end{table}%

\begin{remarks}\label{R-T}~
\begin{enumerate}
\item Other tori are treated in \cite{HHO3}, \cite{BoLe}.
   \item The Robin boundary condition is written $\frac{\partial u}{\partial \nu} + h u = 0$, where $h$ is the Robin parameter, and $\nu$ the outward-pointing normal.
   \item The cylinders considered in the table are $(0,\pi) \times \bS^1_r$, $r \in \set{\frac 12, 1}.$
   \item The flat Klein bottles considered in the table have fundamental domains $(0,\pi)\times (0,c\pi)$, $c \in \set{1,2}$, with sides $\set{x=0}$ and $\set{x=\pi}$ identified, while sides $\set{y=0}$ and $\set{y=c\pi}$ are identified with reversed orientations.
   \item The case of spheres appears in \cite{Ley1996,HHOT2010}.
   \item In the case of Neumann or Robin boundary condition, determining Courant-sharp eigenvalues is more difficult because one can only apply the Faber-Krahn inequality (see Section~\ref{S-upper}) to nodal domains which do not meet the boundary.
 \end{enumerate}
\end{remarks}%

The paper is organized as follows. In Section~\ref{S-moeb}, we describe the M\"{o}bius strip, and compute its spectrum using separation of variables. Sections~\ref{S-esmn}--\ref{S-es23} are devoted to the description of the nodal patterns for the first eigenspaces, $\cE_{\lambda_2}$, $\cE_{\lambda_6}$, and $\cE_{\lambda_7}$. In Section~\ref{S-isop}, we give a Weyl law with a controlled remainder term, and consider isoperimetric and Faber-Krahn inequalities for the M\"{o}bius strip. In Section~\ref{S-upper}, we give an upper bound, together with a condition \`{a} la Faber-Krahn, to be satisfied by Courant-sharp eigenvalues. This section also contains the proof of Theorem~\ref{main}. In Section~\ref{S-eulerm}, we consider an Euler type formula for nodal patterns on the M\"{o}bius strip. In Section~\ref{S-stern}, we give examples of high energy eigenfunctions of $M_1$ with only two nodal domains.\medskip

\textbf{Acknowledgement.} The second author would like to thank C. L\'ena for useful discussions. The authors would also like to thank the two anonymous referees for their constructive comments, and for pointing out several misprints.

\section{The M\"{o}bius strip $M_a$}\label{S-moeb}

\subsection{Presentation and geometry}\label{SS-mop}

Let $S_{\infty} := (0,\pi) \times (-\infty , \infty)$, be the infinite strip with width $\pi$, equipped with the flat metric $dx^2+dy^2$ of $\R^2$. Given $a > 0$, define the following isometries of $S_{\infty}$:
\begin{equation}\label{E-mop-2}
\left\{
\begin{array}{l}
\sigma_a : (x,y) \mapsto (\pi - x,y + a\, \pi)\,,\\[5pt]
\varpi : (x,y) \mapsto (\pi - x,y)\,,\\[5pt]
\tau_t : (x,y) \mapsto (x,y+t), \text{~for~} t \in \R\,.\\[5pt]
\end{array}
\right.
\end{equation}

Define the groups
\begin{equation}\label{E-mop-7}
\left\{
\begin{array}{l}
G := \set{\sigma_a^k ~|~ k \in \Z}\,,\\[5pt]
G_2 := \set{\sigma_a^k ~|~ k \in 2\,\Z}\,.
\end{array}%
\right.
\end{equation}
The group $G_2$ is a subgroup of $G$, of index $2$, generated by $ \sigma_a^2$.\medskip

The action of $G$ on $S_{\infty}$ is smooth, isometric, totally discontinuous, without fixed points. By \cite[Section~2.4]{BeGo1988}, we can consider the quotient manifolds with boundary
\begin{equation}\label{E-mop-CM}
C_a := S_{\infty}/G_{2} \text{~~and~~}M_a := S_{\infty}/G\,,
\end{equation}
 equipped with the flat metric induced from the metric of $S_{\infty}$. \smallskip

The cylinder $C_a$ is the product manifold $(0,\pi) \times \bS_a^1$, where $\bS^1_a$ is the circle $\R/(2\pi a\Z)$. One can view $C_a$ as the rectangle $(0,\pi)\times [0,2\pi a]$ with the sides $(0,\pi) \times \set{0}$ and $(0,\pi) \times \set{2a\pi}$ identified, $(x,0) \sim (x,2a\pi)$. This rectangle is a fundamental domain for the action of $G_2$ on $S_{\infty}$. \smallskip

The isometry $\sigma_a$ induces an isometry of $C_a$, whose square is the identity. The M\"{o}bius strip $M_a$ can also be obtained as the quotient $C_a/\set{\id,\sigma_a}$. One can view $M_a$ as the rectangle $\cR_a := (0,\pi)\times [0,a\pi]$, with the sides $(0,\pi) \times \set{0}$ and $(0,\pi) \times \set{a\pi}$ identified by $(x,0) \sim (\pi - x,a\pi)$. The rectangle $\cR_a$ is a fundamental domain of the action of $G$ on $S_{\infty}$.\smallskip

In the sequel, we will mainly view $M_a$ as $\cR_a$, together with the identification $(x,0) \sim (\pi - x,a\pi)$. \medskip

The isometry $\varpi$ of $S_{\infty}$ induces an isometry of $C_a$ and $M_a$. The action $t \mapsto \tau_t$ induces an isometric action of $\bS_a^1$ on $C_a$ and $M_a$.\medskip

Physically, one can realize $M_a$ by making a paper model: a rectangular sheet of paper, $[0,\pi]\times [0,a\pi]$, is twisted by a rotation of angle $\pi$ about its symmetry axis $\set{x=\pib}$, and the two horizontal sides $\set{y=0}$ and $\set{y=a\pi}$ are glued together.\medskip

The fixed point set of $\varpi$ in $M_a$ is a circle (the ``soul'') of length $a\pi$, while the boundary $\partial M_a$ has length $2a\pi$.\medskip

The cylinder $\cC_a$ embeds isometrically in $\R^3$ by
\begin{equation}\label{E-mop-8}
(0,\pi)\times \bS_a^1 \ni (u,v) \mapsto \left( u, a\, \cos(\frac{v}{a}), a \, \sin(\frac{v}{a}) \right) \in \R^3 \,,
\end{equation}
and we can view the cylinder as a collection of segments $(0,\pi)_a$ attached to the circle $\set{y^2+z^2 = 1} \cap \set{x=0}$, orthogonally to the plane $\set{x=0}$.\medskip

Define the map $F$ by
\begin{equation}\label{E-mop-10}
(w,v) \mapsto \left( w \cos\frac{v}{a}, (R + w \sin\frac{v}{a}) \cos\frac{2v}{a} , (R + w \sin\frac{v}{a}) \sin\frac{2v}{a} \right),
\end{equation}
where $w = u - \frac{\pi}{2}$, $(w,v) \in \left( - \frac{\pi}{2}, \frac{\pi}{2}\right) \times [0,a\pi]$, and $R > 0$. It is easy to see that $F$ is a diffeomorphism from $M_a$ onto $F(M_a) \subset \R^3$ provided that $R > \frac{\pi}{2}$. \medskip

One can view the surface $F(M_a)$ as the collection, $$\set{F\left(
(-\frac{\pi}{2}, \frac{\pi}{2}),\frac{2v}{a}\right) ~\big|~ v \in [0,a\pi]},$$  of segments centered at the point $\left(0,R\cos\frac{2v}{a} ,R\sin\frac{2v}{a}\right)$, contained in the plane $xO\omega_a$, where $\omega_a = \left(0,\cos\frac{2v}{a},\sin\frac{2v}{a}\right)$, and making in this plane an angle $\frac{v}{a}$ with the axis $Ox$.

\begin{remark}\label{R-mop-2}
We will use the map $F$ to visualize the \emph{topology} of the nodal sets and nodal domains of the eigenfunctions of $M_1$ in three dimensions. Note that the map $F$ is not an isometric embedding. As a matter of fact, it is not even conformal, but the vectors $\partial_uF$ and $\partial_vF$ are orthogonal. The surface $M_a$, with the metric $F^{*}(dx^2+dy^2+dz^2)$ induced from the canonical metric in $\R^3$, has negative curvature. According to \cite[\S~15]{HaWe1977} and \cite[Lecture~14]{FuTa2007}, if the M\"{o}bius strip $M_a$ can be isometrically embedded into $\R^3$, then $a > \pib$, and such embeddings actually exist for $a > \sqrt{3}$.
\end{remark}%

\begin{figure}[!ht]
\centering
\includegraphics[scale=0.65]{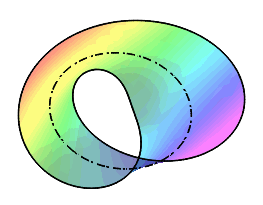}\vspace{-3mm}
\caption{The M\"{o}bius strip embedded in $\R^3$, with boundary and soul}\label{F-moeb-2}
\end{figure}
\FloatBarrier

\subsection{Dirichlet spectrum}\label{SS-mos}

We equip the M\"{o}bius strip, $M_a = S_{\infty}/G$, with the flat metric inherited from $S_{\infty}$, and consider the Dirichlet eigenvalue problem for the associated Laplace-Beltrami operator. The projection $C_a \to M_a$ being a Riemannian covering, the Dirichlet eigenfunctions of $M_a$ can be identified with the Dirichlet eigenfunctions of $C_a$ which are invariant under $\sigma_a$. Since $C_a$ is a product manifold, we can use separation of variables. A complete family of complex eigenfunctions of $C_a$ is given by
\begin{equation}\label{E-mps-2}
\sin(mx) \, \exp(in\frac{y}{a}), m \in \Nb, n \in \Z\,,
\end{equation}
with associated eigenvalues $\lambdah(m,n,a) = m^2 + \frac{n^2}{a^2}$. Here, $\Nb := \set{n \in \Z ~|~ n \ge 1}$. The eigenspace $\cE_{\lambdah(m,n,a)}$, associated with the eigenvalue $\lambdah(m,n,a)$, consists of eigenfunctions of the form
\begin{equation}\label{E-mps-4}
\Phi(x,y) = \sum \alpha_{p,q} \, \sin(px) \, \exp(i q \frac{y}{a})
\end{equation}
where $\alpha_{p,q} \in \C$, and the sum extends over the pairs $(p,q) \in \Nb \times \Z$ such that $\lambdah(p,q,a) = \lambdah(m,n,a)$. Since
\begin{equation*}
\Phi \circ \sigma_a(x,y) = \Phi(\pi - x, y+\pi a) = \sum (-1)^{p+q+1}\, \alpha_{p,q} \, \sin(px) \, \exp(i q \frac{y}{a})\,,
\end{equation*}
it follows that $\Phi \circ \sigma_a = \Phi$ if and only if the summation in \eqref{E-mps-4} extends over the pairs $(p,q) \in \Nb \times \Z$ such that $\lambdah(p,q,a) = \lambdah(m,n,a)$, with the additional condition that  $p+q$ is \emph{odd}. As a consequence, we have the following result.

\begin{lemma}\label{L-mps-2}
A complete family of real Dirichlet eigenfunctions of the M\"{o}bius strip $M_a$, equipped with the flat metric, is
\begin{equation}\label{E-mps-6}
\left\{
\begin{array}{l}
\sin(mx), ~m \in \Nb  \text{~odd, and~}\\[5pt]
\sin(mx) \, \cos(n\frac{y}{a}), ~\sin(mx) \, \sin(n\frac{y}{a}), ~m,n \in \Nb, ~m+n \text{~odd},
\end{array}
\right.
\end{equation}
with associated eigenvalues $\lambdah(m,n,a) = m^2 + \frac{n^2}{a^2}$.
\end{lemma}%

\begin{definition}\label{D-mps-2}
Let $\Phi$ be a Dirichlet eigenfunction of $- \Delta$ on $M_a$. The \emph{nodal set} of $\Phi$ is defined as
\begin{equation}\label{E-mps-ns}
\cZ(\Phi) = \overline{\set{x \in M_a ~|~ \Phi(x)=0}}\,,
\end{equation}
i.e., as the closure in $M_a \cup \partial M_a$ of the set of (interior) zeros of $\Phi$. The \emph{nodal domains} of $\Phi$ are the connected components of $M_a \sm \cZ(\Phi)$.
\end{definition}%

\subsection{Dirichlet spectrum of the M\"{o}bius strip, case $a=1$}\label{SS-mos1}

In the case $a=1$, the Dirichlet eigenvalues of $M_1$ have higher multiplicities. Let $\lambdah(m,n)$ denote $\lambdah(m,n,1)$.

\begin{notation}\label{N-mos-0}
In the sequel, the Dirichlet eigenvalues of the M\"{o}bius strip $M_1$ are denoted,
\begin{equation*}
\lambda_1 < \lambda_2 \le \lambda_3 \le \cdots \,,
\end{equation*}
and listed in nondecreasing order, with multiplicities, starting from the label $1$.
\end{notation}%

The first Dirichlet eigenvalues of $M_1$ are given in the following table, with eigenfunctions given by \eqref{E-mps-6}.

\begin{table}[!htb]
\centering
\resizebox{0.7\textwidth}{!}{%
\begin{tabular}[c]{|c|c|c|c|}%
\hline &&&\\[2pt]
Eigenvalue & $(m,n)$ & $\lambdah(m,n)$ & Multiplicity\\[7pt]
\hline
$\lambda_1$ & $(1,0)$ & $1$ & $1$ \\[5pt]
\hline
$\lambda_2, \lambda_3, \lambda_4, \lambda_5$ & $(1,2), (2,1)$ & $5$ & $4$ \\[5pt]
\hline
$\lambda_6$ & $(3,0)$ & $9$ & $1$ \\[5pt]
\hline
$\lambda_7, \lambda_8, \lambda_9, \lambda_{10}$ & $(2,3), (3,2)$ & $13$ & $4$ \\[5pt]
\hline
$\lambda_{11},\lambda_{12}, \lambda_{13}, \lambda_{14}$ & $(1,4), (4,1)$ & $17$ & $4$ \\[5pt]
\hline
$\lambda_{15}, \lambda_{16}, \lambda_{17}, \lambda_{18}, \lambda_{19}$ & $(3,4), (4,3), (5,0)$ & $25$ & $5$ \\[5pt]
\hline
$\lambda_{20}, \lambda_{21}, \lambda_{22}, \lambda_{23}$ & $(2,5), (5,2)$ & $29$ & $4$ \\[5pt]
\hline
$\lambda_{24}, \lambda_{25}, \lambda_{26}, \lambda_{27}$ & $(1,6), (6,1)$ & $37$ & $4$ \\[5pt]
\hline
$\lambda_{28}, \lambda_{29}, \lambda_{30}, \lambda_{31}$ & $(4,5), (5,4)$ & $41$ & $4$ \\[5pt]
\hline
$\lambda_{32}, \lambda_{33}, \lambda_{34}, \lambda_{35}$ & $(3,6), (6,3)$ & $45$ & $4$ \\[5pt]
\hline
$\lambda_{36}$ & $(7,0)$ & $49$ & $1$ \\[5pt]
\hline
$\lambda_{37}, \lambda_{38}, \lambda_{39}, \lambda_{40}$ & $(2,7), (7,2)$ & $53$ & $4$ \\[5pt]
\hline
$\lambda_{41}, \lambda_{42}, \lambda_{43}, \lambda_{44}$ & $(5,6), (6,5)$ & $61$ & $4$ \\[5pt]
\hline
$\lambda_{45}, \ldots, \lambda_{52}$ & $(1,8), (8,1), (4,7), (7,4)$ & $65$ & $8$\\[5pt]
\hline
\end{tabular}}
\medskip
\caption{First Dirichlet eigenvalues of $M_1$}\label{T-moeb-2}
\end{table}%

\section{Analysis of the first eigenspaces}\label{S-esmn}

When $m, n \in \Nb$ (with $m+n$ odd), the eigenvalue $\lambdah(m,n)=m^2+n^2$ has multiplicity at least $4$. In this case, the corresponding eigenspace, denoted $\cE_{\lambdah(m,n)}$, contains the subspace $\cE_{[m,n]}$ consisting of eigenfunctions of the form
\begin{equation}\label{E-esmn-2}
\left\{
\begin{array}{ll}
\Phi(x,y) = & \sin(mx)\, \left(A \, \cos(ny) + B \, \sin(ny) \right) \\[5pt]
&  \hspace{0.5cm}+ \sin(nx) \left( C \, \cos(my) + D \,  \sin(my)\right).
\end{array}%
\right.
\end{equation}
Note that $\cE_{\lambda_2} = \cE_{[1,2]}$ and $\cE_{\lambda_7} = \cE_{[2,3]}$, while $\cE_{\lambda_{15}} = \cE_{[3,4]} \oplus \cE_{[5,0]}$, see Table~\ref{T-moeb-2}. Since we are interested in nodal sets, we may assume, without loss of generality, that $\Phi$ is \emph{normalized} by $A^2+B^2+C^2+D^2 = 1$. The function in the first parenthesis on the right hand side of \eqref{E-esmn-2} can be written as $\sqrt{A^2 + B^2} \, \sin(ny+\alpha)$, and the function in the second parenthesis as $\sqrt{C^2+D^2}\, \sin(my+\beta)$, for some $\alpha,\beta \in (-\pi,\pi]$. Defining $\theta \in [0,\pib]$ by,
\begin{equation*} 
\left\{
\begin{array}{l}
\cos \theta = \sqrt{A^2 + B^2}\,,\\[5pt]
\sin \theta = \sqrt{C^2 + D^2}\,,
\end{array}
\right.
\end{equation*}
we write $\Phi$ as,
\begin{equation}\label{E-esmn-8}
\begin{array}{ll}
\Phi_{\beta,\theta}(x,y) = & \cos\theta \, \sin(mx) \, \sin(ny+\alpha)\\[5pt]
& \hspace{0.5cm} + \sin\theta \, \sin(nx) \, \sin(my+\beta)\,.
\end{array}%
\end{equation}
Using the isometric $\bS^1$ action $\tau_t : (x,y) \mapsto (x,y+t)$ on $M_1$, we can assume that $\alpha=0$. We also observe that
\begin{equation*}
\begin{array}{ll}
\Phi_{\beta+\pi,\theta}(x,y) = & (-1)^n\, \cos\theta\, \sin(mx)\, \sin(n(y+\pi)) \, +\\[5pt]
& ~~ (-1)^{m+1}\, \sin\theta\, \sin(nx)\, \sin(m(y+\pi)+\beta)\,,
\end{array}%
\end{equation*}
so that
\begin{equation}\label{E-esmn-10}
\Phi_{\beta+\pi,\theta}(x,y) = (-1)^n\, \Phi_{\beta,\theta}(x,y+\pi) =
(-1)^n \, \Phi_{\beta,\theta}(\pi - x,y)\,.
\end{equation}
It follows that the nodal sets of $\Phi_{\beta+\pi,\theta}$ and $\Phi_{\beta,\theta}$ are symmetric with respect to $\set{x=\pib}$ in $\cR_1$ or $M_1$.

\begin{proposition}\label{P-esmn-2}
To determine the nodal patterns of the Dirichlet eigenfunctions $\Phi \in \cE_{[m,n]}$, for $m,n \in \Nb$ with $m+n$ is odd, it is sufficient to study the nodal properties of the family $\Phi_{\beta,\theta}$,
\begin{equation*}
\Phi_{\beta,\theta}(x,y) = \cos \theta \, \sin(mx) \, \sin(ny) + \sin \theta \, \sin(nx) \, \sin(my+\beta)\,,
\end{equation*}
for $\beta \in (0,\pi]$, and $\theta \in [0,\pib]$.\smallskip

When $[m,n]=[1,2]$ or $[2,3]$, one can reduce the parameter set further.
\begin{enumerate}
  \item For $[m,n]=[1,2]$, it suffices to consider $(\beta,\theta) \in (0,\pib] \times [0,\pib]$. Indeed,
\begin{equation}\label{E-esmn-es12}
\left\{
\begin{array}{l}
\Phi_{\beta+\pib,\theta} = - \, \Phi_{\beta,\theta} \circ \tau_{-\pib}\,,\\[5pt]
\cZ(\Phi_{\beta+\pib,\theta}) = \tau_{\pib} \cZ(\Phi_{\beta,\theta})\,.
\end{array}
\right.
\end{equation}
  \item For $[m,n]=[2,3]$, it suffices to consider $(\beta,\theta) \in (0,\pic] \times [0,\pib]$. Indeed,
\begin{equation}\label{E-esmn-es23}
\left\{
\begin{array}{l}
\Phi_{\beta+\pic,\theta} = - \, \Phi_{\beta,\theta} \circ \tau_{-\pic}\,,\\[5pt]
\cZ(\Phi_{\beta+\pic,\theta}) = \tau_{\pic} \cZ(\Phi_{\beta,\theta})\,.
\end{array}
\right.
\end{equation}
\end{enumerate}
\end{proposition}%

For the analysis of nodal sets, we need the following definition.

\begin{definition}\label{D-esmn-2}
A point $(x,y)$ is a \emph{critical zero} of a function $\Phi$ if $\Phi(x,y)=0$ and $d_{(x,y)}\Phi=0$. A critical zero has \emph{order} $k \in \Nb$ if the function $\Phi$ and its derivatives of order less than or equal to $(k-1)$ vanish at $(x,y)$, and a least one derivative of order $k$ does not. A point $(x,y)$ such that $\Phi(x,y)=0$ and $d_{(x,y)}\Phi \not=0$ is called \emph{regular}.
\end{definition}%

For eigenfunctions in dimension $2$, the critical zeros are isolated, and their orders determine the structure of the nodal set locally. In the case of the M\"{o}bius strip $M_1$, the eigenfunctions are defined globally on $\R^2$, and one defines a boundary critical zero as a critical zero of (the extended function) $\Phi$, which lies on the boundary $\partial M_1$.

As a first step to Theorem~\ref{main}, we prove the following result.

\begin{proposition}\label{P-esmn-4}
For the M\"{o}bius strip $M_1$, the Dirichlet eigenvalues $\lambda_1$ and $\lambda_2$ are Courant-sharp. The eigenvalues $\lambda_3, \ldots, \lambda_{10}$ are not Courant-sharp.
\end{proposition}%

\pf The first assertion is a direct consequence of Courant's nodal domain theorem \cite[Chap.~{VI}.6]{CH1953}.\smallskip

According to Remark~\ref{R-int-2} and Table~\ref{T-moeb-2},  the eigenvalue $\lambda_2 = \lambdah(1,2)$ has multiplicity $4$, so that $\lambda_3, \ldots, \lambda_5$ cannot be Courant-sharp. Since $\lambda_6 = \lambdah(3,0)$ is simple, and since the associated eigenfunction $\sin(3x)$ has $2$ nodal domains in $M_1$, the eigenvalue $\lambda_6$ is not Courant-sharp. The eigenvalue $\lambda_7 = \lambdah(2,3)$ has multiplicity $4$. This implies that $\lambda_8, \ldots, \lambda_{10}$ cannot be Courant-sharp.\smallskip

To finish the proof of Proposition~\ref{P-esmn-4}, it suffices to prove that $\lambda_7$ is not Courant-sharp.  This is the purpose of Lemma~\ref{L-es23-2} in Section~\ref{S-es23}; the proof is by analyzing the possible nodal patterns of eigenfunctions in $\cE_{\lambda_7} = \cE_{[2,3]}$. \hfill \qed \medskip

Although we already know that $\lambda_2$ is Courant-sharp, it is interesting to describe the possible nodal patterns of associated eigenfunctions. They indeed present some peculiar properties when compared to second Dirichlet eigenfunctions of a simply-connected domain. In Section~\ref{S-es12}, we give the possible nodal patterns of eigenfunctions in $\cE_{\lambda_2} = \cE_{[1,2]}$, up to isometries given by the $\bS^1$ action, and to the symmetry with respect to $\set{x=\pib}$. \medskip

In Section~\ref{S-es23}, we give the possible nodal patterns of eigenfunctions in $\cE_{\lambdah(2,3)}$, up to isometries given by the $\bS^1$ action, and to the symmetry with respect to $\set{x=\pib}$. As a consequence, we shall conclude that an eigenfunction in $\cE_{\lambda_7}$ has at most $6$ nodal domains, so that $\lambda_7$ is not Courant-sharp.

\section{Analysis of the eigenspace $\cE_{\lambda_2}$}\label{S-es12}

In this section, we describe the nodal patterns of the eigenfunctions $\Phi_{\beta,\theta} \in \cE_{[1,2]} = \cE_{\lambda_2}$, with $\beta \in (0,\pib]$ and $\theta \in [0,\pib]$.  According to Proposition~\ref{P-esmn-2},
this is sufficient to determine all the possible nodal patterns, up to isometries. We consider the following cases and subcases.
\begin{enumerate}
  \item $\theta=0$ and $\theta=\pib$.\smallskip
  \item $\beta=\pib$, three subcases: $0 < \theta < \pid$, $\theta=\pid$ and $\pid < \theta < \pib$.\smallskip
  \item $\beta \in (0,\pib)$, three subcases $0 < \theta < \theta_{\beta}$, $\theta=\theta_{\beta}$ and $\theta_{\beta} < \theta < \pib$\\
      (the function $\beta \mapsto \theta_{\beta}$ is described below).
\end{enumerate}

The analysis of these cases can be done using the methods described in Section~\ref{S-es23}. For this reason, we shall not give full details here, but a mere description. \medskip

The figures below display the nodal lines in the fundamental domain $\cR_1$; recall that the M\"{o}bius strip $M_1$ is obtained by identifying the lines $\set{y=0}$ and $\set{y=\pi}$ via $(x,y) \sim (\pi-x,y+\pi)$. The nodal lines appear in \emph{red}, the Dirichlet boundary in \emph{blue}. When the lines $\set{y=0}$ and $\set{y=\pi}$ are not nodal, they appear as \emph{dashed black lines} to indicate the M\"{o}bius identification.

\subsection{Case $\theta=0$ or $\theta=\pib$}\label{SS-es12-sc1-2}

In this case, the eigenfunctions are decomposed and the nodal sets explicit, see Figure~\ref{F-fig12-sc1}. When $\theta=0$, Figure~(A), there are two disjoint nodal lines (in $M_1$), hitting the boundary at critical zeros of order $2$, and no interior critical zero. When $\theta=\pib$, Figure~(B) \& (C), there are two nodal lines which intersect at an interior critical zero of order $2$, and hit the boundary at critical zeros of order $2$. As expected, in each case, there are two nodal domains.

\begin{figure}[h]
\centering
\begin{subfigure}[t]{.30\textwidth}
\centering
\includegraphics[width=\linewidth]{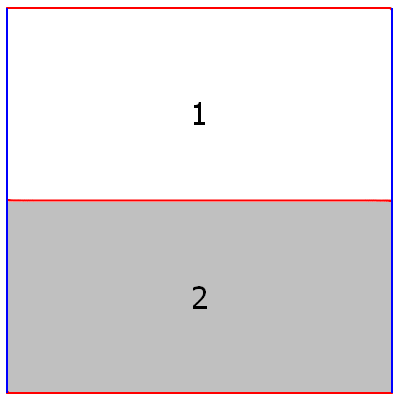}
\caption{$\Phi_{\beta,0}$}
\end{subfigure}
\begin{subfigure}[t]{.30\textwidth}
\centering
\includegraphics[width=\linewidth]{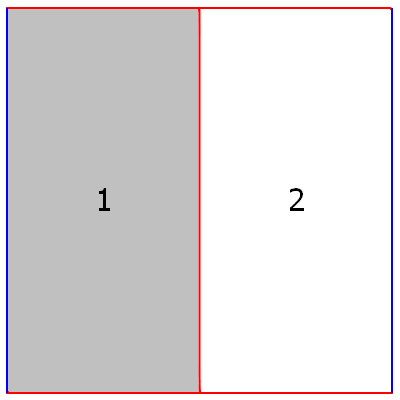}
\caption{$\Phi_{0,\pib}$}
\end{subfigure}
\begin{subfigure}[t]{.30\textwidth}
\centering
\includegraphics[width=\linewidth]{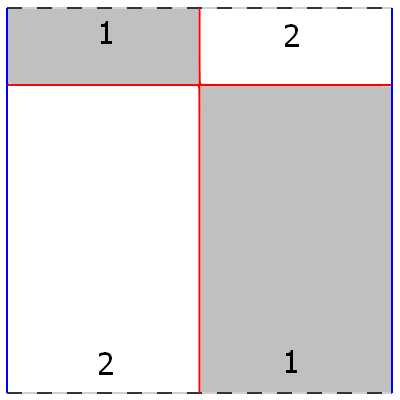}
\caption{$\Phi_{\frac{\pi}{5},\pib}$}
\end{subfigure}
\caption{Nodal sets $\cZ(\Phi_{\beta,0})$, $\cZ(\Phi_{0,\pib})$, and $\cZ(\Phi_{\frac{\pi}{5},\pib})$}\label{F-fig12-sc1}
\end{figure}
\FloatBarrier

The M\"{o}bius strip $M_1$ is not simply-connected (it retracts onto the soul circle). Figures~\ref{F-fig12-sc1}~(B) and (C) display second Dirichlet eigenfunctions with an interior critical zero, and a nodal set which contains a closed curve, the arc $\set{x=\pib}$. We refer to Figure~\ref{F-fig12-sc2-aW} for a 3D-picture. The existence of interior critical zeros is a consequence of the multiplicity of $\lambda_2$ which is $4$ for $M_1$. For domains $\Omega \subset \R^2$, P\'{o}lya's nodal line conjecture states that a second Dirichlet eigenfunction cannot have a closed nodal line. The conjecture has been proved for convex domains by A.~Melas (smooth convex domains) and G.~Alessandrini (general convex domains). A consequence of the conjecture is the non-existence of interior critical zeros. A counter-example to the conjecture has been constructed by M.~Hoffmann-Ostenhof, T.~Hoffmann-Ostenhof and N.~Nadirashvili, with a domain $\Omega$ not simply connected. We refer to \cite{HO2Na1997} and the recent paper \cite{Ken2018} for more details and references on the nodal line conjecture.

\subsection{Case $\beta=\pib$ and $0 < \theta < \pib$}\label{SS-es12-sc4}

In this case, a bifurcation occurs at $\theta = \pid$. The different patterns are illustrated in Figure~\ref{F-fig12-sc4}. For $\theta = \pid$, Figure~(B), the line $\set{y=\pib}$ is contained in the nodal set $\cZ(\Phi_{\pib,\pid})$. It hits the boundary part $\set{x=0}$ at a critical zero of order $2$, and the boundary part $\set{x=\pi}$ at a critical zero of order $4$. There is another nodal component hitting the boundary at this point, a closed curve. When $0 < \theta < \pid$, Figure~(A), the nodal set consists in two disjoint simple curves with end points critical zeros of order $2$ on the boundary. No interior critical zeros for both patterns. When $\pid < \theta < \pib$, Figure~(C), the nodal set consists in two simple curves which intersect at an interior critical zero of order $2$; one curve is closed, the other hits the boundary at two critical zeros of order $2$. As expected, there are two nodal domains in each cases. We refer to Figure~\ref{F-fig12-sc4-bW} for a 3D-picture.

\begin{figure}[ht]
\centering
\begin{subfigure}[t]{.30\textwidth}
\centering
\includegraphics[width=\linewidth]{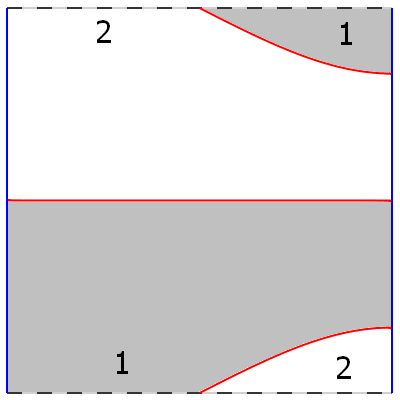}
\caption{$0 < \theta < \pid$}
\end{subfigure}
\begin{subfigure}[t]{.30\textwidth}
\centering
\includegraphics[width=\linewidth]{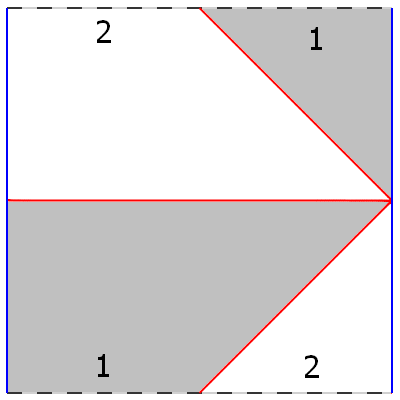}
\caption{$\theta = \pid$}
\end{subfigure}
\begin{subfigure}[t]{.30\textwidth}
\centering
\includegraphics[width=\linewidth]{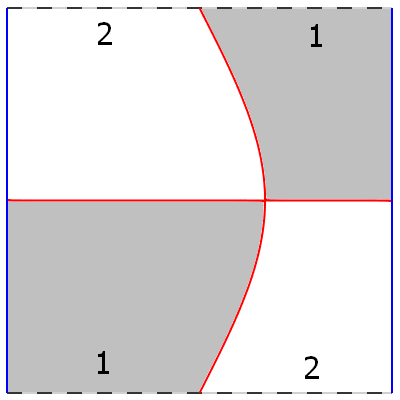}
\caption{$\pid < \theta < \pib$}
\end{subfigure}
\caption{Nodal sets $\cZ(\Phi_{\pib,\theta})$, for $0 < \theta < \pib$}\label{F-fig12-sc4}
\end{figure}
\FloatBarrier

\subsection{Case $\beta=0$ and $0 < \theta < \pib$}\label{SS-es12-sc3}

This case is similar to the preceding one. Indeed, recall that the nodal sets $\cZ(\Phi_{0,\theta})$ and $\cZ(\Phi_{\pib,\theta})$ are isometric. We give the corresponding pictures for completeness. \medskip

\begin{figure}[h]
\centering
\begin{subfigure}[t]{.30\textwidth}
\centering
\includegraphics[width=\linewidth]{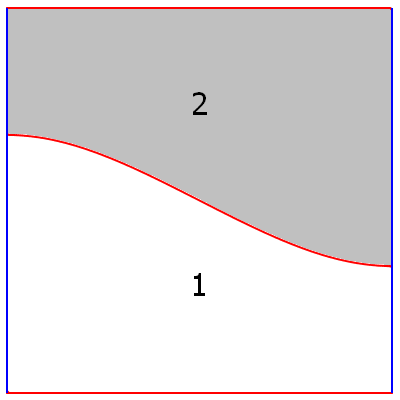}
\caption{$0 < \theta < \pid$}
\end{subfigure}
\begin{subfigure}[t]{.30\textwidth}
\centering
\includegraphics[width=\linewidth]{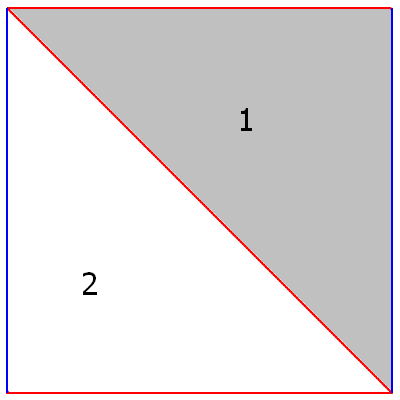}
\caption{$\theta = \pid$}
\end{subfigure}
\begin{subfigure}[t]{.30\textwidth}
\centering
\includegraphics[width=\linewidth]{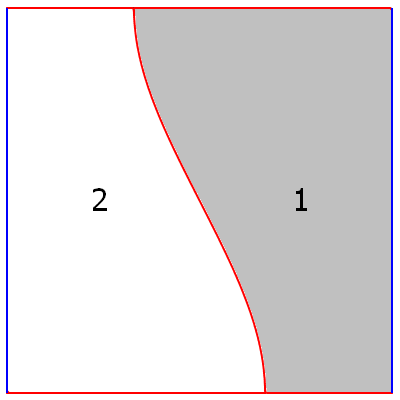}
\caption{$\pid < \theta < \pib$}
\end{subfigure}
\caption{Nodal sets $\cZ(\Phi_{0,\theta})$, for $0 < \theta < \pib$}\label{F-fig12-sc3}
\end{figure}

\FloatBarrier

\subsection{Case $0 < \beta < \pib$ and $0 < \theta < \pib$}\label{SS-es12-sc6-gc}

This case presents some novelty. Indeed, given $\beta \in (0,\pib)$, there exists a unique $\theta_{\beta} \in (0,\pib)$ such that the nodal sets of the family $\{\Phi_{\beta,\theta}\}$, for fixed $\beta$ and $0 < \theta < \pib$, present a bifurcation at $\theta = \theta_{\beta}$. More precisely, given $\beta \in (0,\pib)$, there exist a unique $y_{\beta} \in (0,\pib)$, given by $\cot^3(y_{\beta}) = \cot(\beta)$, and a unique value $\theta_{\beta} \in (0,\pib)$, given by $\cot(\theta_{\beta}) = 2 \, \frac{\sin(y_{\beta}+\beta)}{\sin(2y_{\beta})}$, such that the following description holds. \medskip

Fix $\beta \in (0,\pib)$.  When $\theta=\theta_{\beta}$, Figure~\ref{F-fig12-sc6-gc-1}~(B), the nodal set $\cZ(\Phi_{\beta,\theta_{\beta}})$ hits the boundary part $\set{x=\pi}$ at the point $(\pi,y_{\beta})$, which is a critical zero of order $3$. The nodal set consists in two curves issued from this point, forming equal angles with the boundary; one of them hits the boundary part $\set{x=0}$, the other hits the boundary part $\set{x=\pi}\times (\pib,\pi)$, at critical zeros of order $2$. They do not intersect in the interior. When $0 < \theta < \theta_{\beta}$, Figure~(A), the nodal sets consists in two disjoint curves, with end points on the boundary, critical zeros of order $2$. When $\theta_{\beta} < \theta < \pib$, Figure~(C), the nodal set consists of one simple curve with end points on the boundary, critical zeros of order $2$. No interior critical zeros in these three patterns. As expected there are two nodal domains in all these cases. \medskip

Figure~\ref{F-fig12-sc6-gc-1}~(C) presents another peculiarity. The nodal domain labelled ``1'' is not simply connected, and actually homeomorphic to a M\"{o}bius strip, and it is not orientable. We refer to Figure~\ref{F-fig12-sc6-cW} for 3D-pictures.

\begin{figure}[h]
\centering
\begin{subfigure}[t]{.30\textwidth}
\centering
\includegraphics[width=\linewidth]{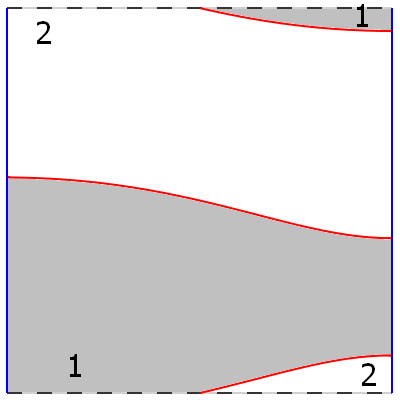}
\caption{$0 < \theta < \theta_{\beta}$}
\end{subfigure}
\begin{subfigure}[t]{.30\textwidth}
\centering
\includegraphics[width=\linewidth]{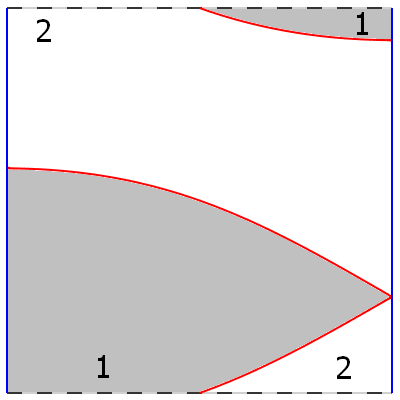}
\caption{$\theta = \theta_{\beta}$}
\end{subfigure}
\begin{subfigure}[t]{.30\textwidth}
\centering
\includegraphics[width=\linewidth]{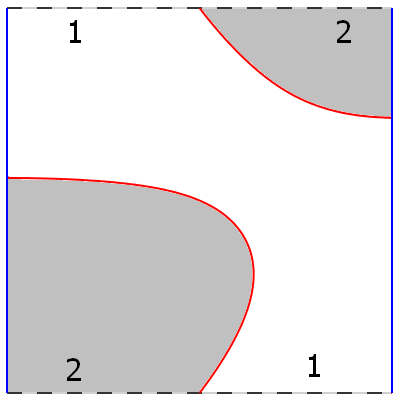}
\caption{$\theta_{\beta} < \theta < \pib$}
\end{subfigure}
\caption{Nodal set $\cZ(\Phi_{\beta,\theta})$, for $0 < \beta < \pib$ and
$0 < \theta < \pib$}\label{F-fig12-sc6-gc-1}
\end{figure}
\FloatBarrier \medskip

\textbf{Some 3D-pictures.}\medskip

\begin{figure}[!ht]
\centering
\includegraphics[scale=0.40]{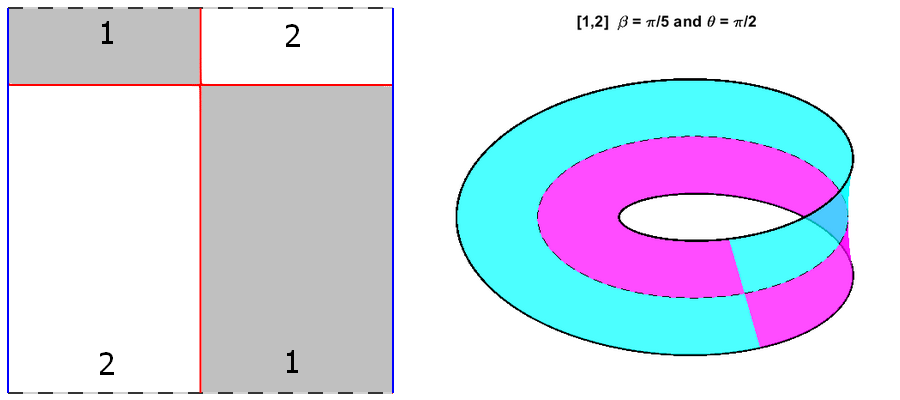}
\caption{Second eigenfunction with an interior critical zero and a closed nodal line ($\beta=\frac{\pi}{5}, \theta=\pib$)}\label{F-fig12-sc2-aW}
\end{figure}

\begin{figure}[!ht]
\centering
\includegraphics[scale=0.40]{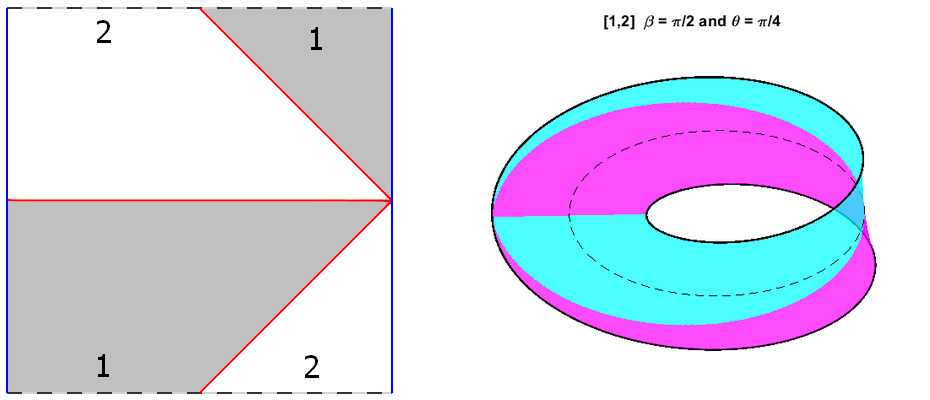}
\caption{Second eigenfunction with a boundary critical zero of order $4$ ($\beta=\pib, \theta=\pid$)}\label{F-fig12-sc4-bW}
\end{figure}

\begin{figure}[!ht]
\centering
\includegraphics[scale=0.40]{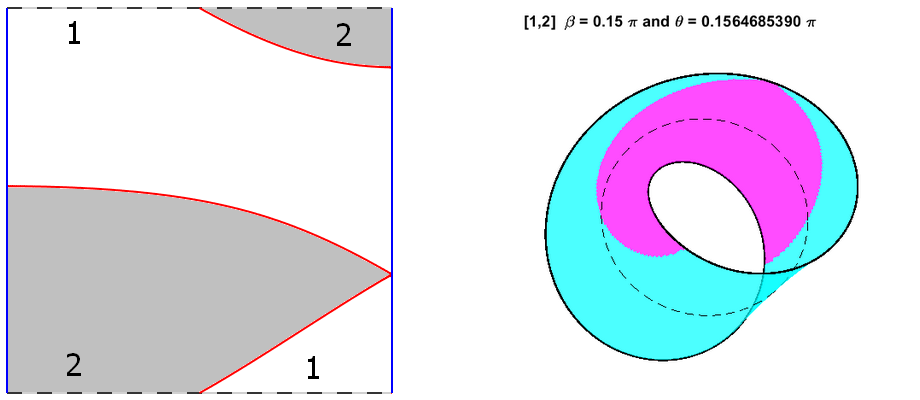}
\caption{Second eigenfunction with a boundary critical zero of order $3$ ($\beta$ given, $\theta=\theta_{\beta}$)}\label{F-fig12-sc6-bW}
\end{figure}

\begin{figure}[!ht]
\centering
\includegraphics[scale=0.40]{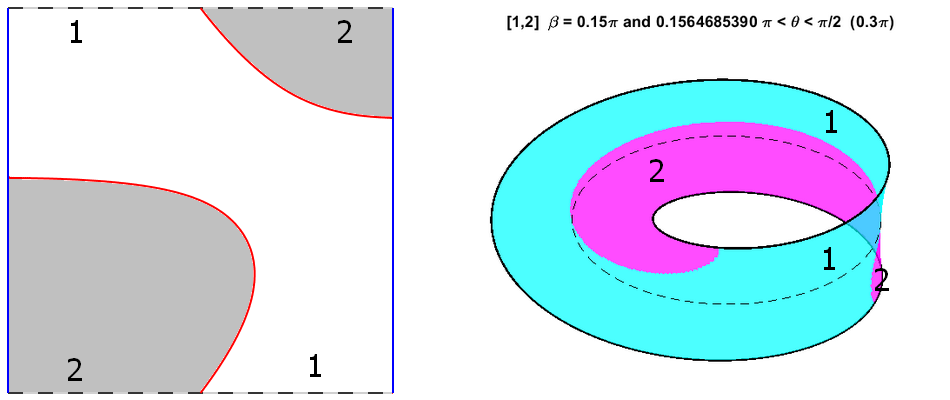}
\caption{Second eigenfunction with a non simply connected nodal domain ($\beta$ given, $\theta_{\beta} < \theta < \pib$)}\label{F-fig12-sc6-cW}
\end{figure}

\FloatBarrier

\begin{remark}\label{R-es12-conf}
The picture on the right hand side of Figure~\ref{F-fig12-sc4-bW}, resp.  \ref{F-fig12-sc6-cW}, seems to violate the equal angle condition at the boundary critical zero of order $4$, resp $3$. The reason is that the map $F$ is not conformal, as pointed out in Remark~\ref{R-mop-2}.
\end{remark}%

\section{Analysis of the eigenspace $\cE_{\lambda_7}$}\label{S-es23}

The purpose of this section is to finish the proof of Proposition~\ref{P-esmn-4} by proving the following lemma. Recall that $\cE_{\lambda_7} = \cE_{[2,3]}$, see Table~\ref{T-moeb-2}.

\begin{lemma}\label{L-es23-2}
An eigenfunction $\Phi$ associated with the Dirichlet eigenvalue $\lambda_7 $ of the M\"{o}bius strip $M_1$ has at most six nodal domains. As a consequence, the eigenvalue $\lambda_7$ is not Courant-sharp.
\end{lemma}%

Taking Proposition~\ref{P-esmn-2} into account, we consider the family of eigenfunctions,
\begin{equation}\label{E-es23-102}
\Phi_{\beta,\theta}(x,y) = \cos\theta \, \sin(2x) \sin(3y) + \sin\theta \, \sin(3x) \, \sin(2y+\beta)\,,
\end{equation}
for $\beta \in [0,\pic]$ and $\theta \in [0,\pib]$, and examine three cases. More precisely, we will prove that,
\begin{enumerate}
  \item when $\theta = 0$ or $\theta = \pib$, $\Phi_{\beta,\theta}$ has $6$ nodal domains, see Subsection~\ref{SS-sc23-12};
  \item when $\beta \in (0,\pic)$ and $\theta \in (0,\pib)$, $\Phi_{\beta,\theta}$ has $3$ nodal domains, see Subsection~\ref{SS-sc23-6};
  \item when $\beta \in \set{0,\pic}$ and $\theta \in (0,\pib)$, $\Phi_{\beta,\theta}$ has $4$ nodal domains, see Subsection~\ref{SS-sc23-34}.
\end{enumerate}

Subsections~\ref{SS-sc23-gp} and \ref{SS-es23-f-g} contain technical preliminaries.
\medskip

Recall that the M\"{o}bius strip $M_1$ is obtained by identifying the lines $\set{y=0}$ and $\set{y=\pi}$ via $(x,y) \sim (\pi-x,y+\pi)$. The figures below represent the nodal sets in the fundamental domain $\cR_1$, or on the topological 3D-representation given by the map $F$ defined in \eqref{E-mop-10}, see Remark~\ref{R-mop-2}. The nodal lines appear in \emph{red} in the fundamental domain $\cR_1$, the Dirichlet boundary in \emph{blue}. When $\set{y=0} \sim \set{y=\pi}$ is not a nodal line, the \emph{dashed black lines} indicate the M\"{o}bius identification.

\subsection{Special values $\theta=0$ and $\theta=\pib$}\label{SS-sc23-12}

The eigenfunctions $\Phi_{\beta,0}$ (actually independent of $\beta$), and $\Phi_{\beta,\pib}$ play a special role. Their nodal sets are explicit; they are the unions of horizontal and vertical segments, see Figure~\ref{F-23-sc1}. There are both interior and boundary critical zeros, all of order $2$.  \emph{There are $6$ nodal domains in each case.}\medskip

\begin{figure}[ht]
\centering
\begin{subfigure}[t]{.80\textwidth}
\centering
\includegraphics[scale=0.3]{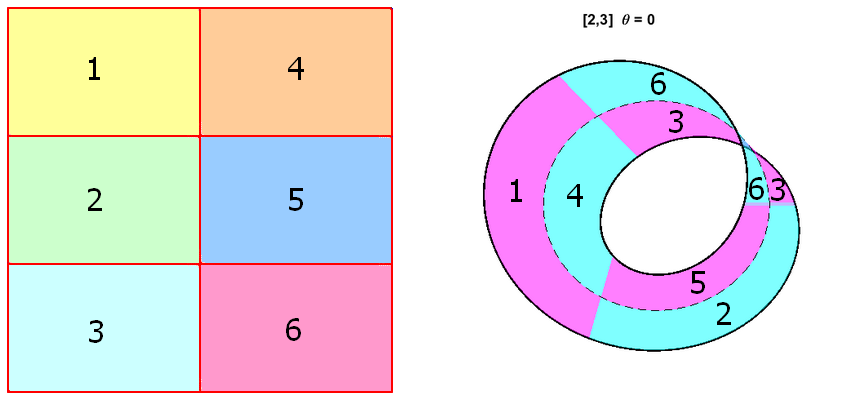}
\caption{$\Phi_{\beta,0}$}
\end{subfigure}
\begin{subfigure}[t]{.80\textwidth}
\centering
\includegraphics[scale=0.3]{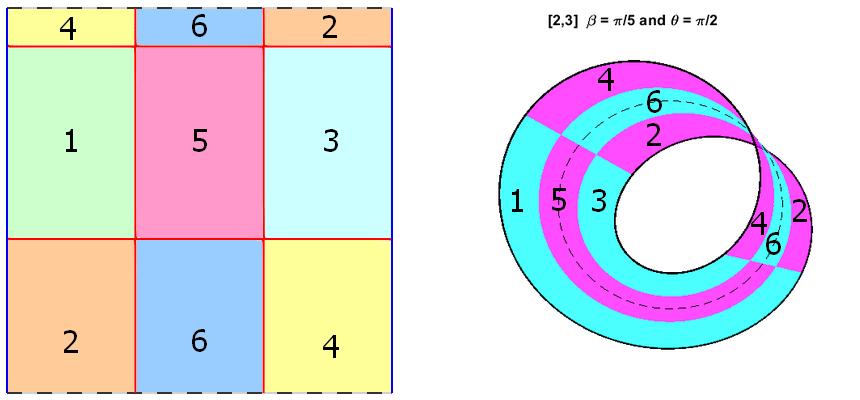}
\caption{$\Phi_{\frac{\pi}{5},\pib}$}
\end{subfigure}
 \caption{Nodal sets $\cZ(\Phi_{\beta,0})$ and $\cZ(\Phi_{\frac{\pi}{5},\pib})$}\label{F-23-sc1}
\end{figure}

Figure~(A) illustrates the following property: cutting the M\"{o}bius strip along the soul circle yields a strip with half the width, and twice the length of the original strip. This strip is doubly twisted and hence orientable. Figure~(B) illustrates what happens when one cuts the strip in the $y$-direction, at exactly one-third the width. This operation produces two intertwined strips, one M\"{o}bius strip (corresponding to the domains labelled ``5'' and ``6''), and a doubly twisted (orientable) strip (corresponding to the domains labelled ``1, 2, 3'' and ``4'').
\FloatBarrier

\subsection{General properties of the nodal sets}\label{SS-sc23-gp}

As a preparation for the analysis of the nodal sets of $\Phi_{\beta,\theta}$, we gather some general properties.

\begin{proposition}\label{P-es23-2}
Assume that $\beta \in [0,\pic]$ and $\theta \in (0,\pib)$. The following properties hold.
\begin{enumerate}
  \item For any $\xi \in (0,\pi)$, the nodal set $\cZ(\Phi_{\beta,\theta})$ does not contain the segment $\set{x=\xi}$.
  \item For $\eta \in [0,\pi]$, the nodal set $\cZ(\Phi_{\beta,\theta})$ contains the horizontal segment $\set{y=\eta}$, if and only if $(\eta,\beta) \in \set{(0,0);(\pi,0);(\pic,\pic)}$.
\end{enumerate}
\end{proposition}%

\pf In this proof, we use the abbreviation\footnote{We shall do so in the subsequent proofs, whenever there is no ambiguity.} $\Phi$ for $\Phi_{\beta,\theta}$.\smallskip

\emph{Assertion~(1).}~ Let $\xi \in (0,\pi)$. Then, $\set{x=\xi} \subset \cZ(\Phi_{\beta,\theta})$ if and only if
\begin{equation*}
\Phi(\xi,y) = \cos\theta \, \sin(2\xi) \sin(3y) + \sin\theta \, \sin(3\xi) \, \sin(2y+\beta) \equiv 0\,,
\end{equation*}
as a function of $y$.\medskip

\noid If $\beta=0$, choosing $y=\pib$ we find that $\sin(2\xi)=0$. This in turn implies that $\sin\theta \, \sin(3\xi) \, \sin(2y+\beta) \equiv 0$, and hence that $\sin(3\xi)=0$, a contradiction.\smallskip

\noid If $\beta \not = 0$, choosing $y=0$, we find that $\sin(3\xi)=0$. This in turn implies that $ \cos\theta \, \sin(2\xi) \sin(3y) \equiv 0$, and hence that $\sin(2\xi)=0$, a contradiction.\medskip

\emph{Assertion~(2).}~ Take $\eta \in [0,\pi]$, and assume that $\Phi(x,\eta) \equiv 0$ as a function of $x$. Choosing $x=\pib$, we find that $\sin(3\eta)=0$. This in turn implies that $\sin\theta \, \sin(3x) \, \sin(2\eta+\beta) \equiv 0$, which implies that $\sin(2\eta+\beta)=0$. The conditions $\sin(3\eta)=0$ and $\sin(2\eta+\beta)=0$ occur simultaneously if and only if $(\eta,\beta) \in \set{(0,0);(\pi,0);(\pic,\pic)}$. \hfill \qed

\begin{proposition}\label{P-es23-4}
Assume that $\beta \in [0,\pic]$, and that $\theta \in (0,\pib)$. Let $(\xi,\eta) \in \overline{M_1}$ (or $\overline{\cR_1}$) be a critical zero of $\Phi_{\beta,\theta}$. The following properties hold.
\begin{enumerate}
  \item The point $(\xi,\eta)$ satisfies,
  \begin{equation}\label{E-es23-106}
    \sin(3\eta) \, \sin(2\eta+\beta) \, \sin(\xi) = 0.
  \end{equation}
  \item If $\xi\in (0,\pi)$, i.e., if $(\xi,\eta) \in M_1$ is an interior critical zero of $\Phi_{\beta,\theta}$, then
  \begin{equation}\label{E-es23-108}
  (\eta,\beta) \in \set{(0,0);(\pi,0);(\pic,\pic)}\,,
  \end{equation}
  and the point $(\xi,\eta)$ has order $2$. In particular, if $\beta \in (0,\pic)$ and $\theta \in (0,\pib)$, the function $\Phi_{\beta,\theta}$ has no interior critical zero. If $\beta \in \set{0,\pic}$ and $\theta \in (0,\pib)$, any interior critical zero of $\Phi_{\beta,\theta}$ lies on the line $\set{y=\beta}$.\smallskip
  \item If $\xi \in \set{0,\pi}$, i.e., if $(\xi,\eta)$ is a boundary critical zero of $\Phi_{\beta,\theta}$, then $(\xi,\eta)$ has order $k \in \set{2,3,4}$. More precisely,
      \begin{enumerate}
        \item the point $(\xi,\eta)$ is a boundary critical zero if $\xi \in \set{0,\pi}$, and
$$
\left\{
\begin{array}{l}
\text{either~} (\eta,\beta) \in \set{(0,0);(\pi,0);(\pic,\pic)},\\[5pt]
\text{or~} \eta \not \in \set{0,\pic,\pit,\pi} \text{~and~} \cot\theta =
- \cos(\xi)\, f(\beta,\eta),
\end{array}%
\right.
$$
where the function $f(\beta,y)$ is defined by,
\begin{equation}\label{E-es23-108f}
f(\beta,y) =  \frac 32 \, \frac{\sin(2y+\beta)}{\sin(3y)}\,~~~ \text{~for~} \beta \in [0,\pic], ~y \in (0,\pi)\sm\set{\pic,\pit}\,;
\end{equation}
        \item the boundary critical zero $(\xi,\eta)$ has order at least $3$ if and only if \goodbreak $g(\beta,\eta)=0$ as well, where the function $g(\beta,y)$ is defined on $[0,\pi]^2$ by
\begin{equation}\label{E-es23-108g}
g(\beta,y) = 2 \cos(2y+\beta) \sin(3y) - 3 \cos(3y) \sin(2y+\beta)\,;
\end{equation}
        \item the boundary critical zero $(\xi,\eta)$ can only have order at least $4$ if \goodbreak $\sin(3\eta) \sin(2\eta+ \beta) = 0$. In that case, the order is $4$.
      \end{enumerate}
\end{enumerate}
\end{proposition}%

\pf We use the abbreviation $\Phi$ for $\Phi_{\beta,\theta}$ in the proof.\smallskip

\emph{Assertion~(1).}~ The point $(\xi,\eta)$ is a critical zero of $\Phi$ if and only if it satisfies the system of equations
\begin{subequations}\label{E-es23-110}
\begin{align}
\Phi(x,y) &=
\cos\theta \, \sin(2x) \sin(3y) + \sin\theta \, \sin(3x) \sin(2y+\beta) =0, \label{E-es23-110n}\\[5pt]
\partial_x \Phi(x,y) &=
2 \cos\theta \, \cos(2x) \sin(3y) + 3 \sin\theta \, \cos(3x) \sin(2y+\beta) =0, \label{E-es23-110x}\\[5pt]
\partial_y \Phi(x,y) &=
3 \cos\theta \, \sin(2x) \cos(3y) + 2 \sin\theta \, \sin(3x) \cos(2y+\beta) =0.
\label{E-es23-110y}
\end{align}
\end{subequations}%

Since $\theta \in (0,\pib)$, if $(\xi,\eta)$  is a critical zero of $\Phi$, the determinant of the linear system \eqref{E-es23-110n} and \eqref{E-es23-110x} must vanish at $(\xi,\eta)$,
\begin{equation*}
\sin(3\eta)  \sin(2\eta+\beta) \, \left( 3 \sin(2\xi) \, \cos(3\xi) - 2 \cos(2\xi) \, \sin(3\xi)\right) = 0,
\end{equation*}
or equivalently
\begin{equation*}
\sin(3\eta) \, \sin(2\eta+\beta) \, \sin^3(\xi) \, (4\cos^2(\xi) +1)= 0\,.
\end{equation*}

\emph{Assertion~(2).}~ If $\xi \in (0,\pi)$, \eqref{E-es23-106} implies that $\sin(3\eta)\, \sin(2\eta+\beta) = 0$. Since $(\xi,\eta)$ is an interior critical zero, the system \eqref{E-es23-110}, and the assumption $\theta \in  (0,\pib)$, imply that
\begin{equation}\label{E-es23-112}
\sin(3\eta) \sin(2\eta+\beta) = 0 \Leftrightarrow (\eta,\beta) \in \set{(0,0);(\pi,0);(\pic,\pic)}\,.
\end{equation}
To conclude the proof of the assertion, we need the second derivatives of $\Phi$.
\begin{subequations}\label{E-es23-120}
\begin{align}
\partial_{xx} \Phi(x,y) &= - 4 \cos\theta \, \sin(2x) \sin(3y) - 9 \sin\theta \, \sin(3x) \sin(2y+\beta),\label{E-es23-120xx}\\[5pt]
\partial_{xy} \Phi(x,y) &= 6 \cos\theta \, \cos(2x) \cos(3y) + 6 \sin\theta \, \cos(3x) \cos(2y+\beta),\label{E-es23-120xy}\\[5pt]
\partial_{yy} \Phi(x,y) &= -9 \cos\theta \, \sin(2x) \sin(3y) - 4 \sin\theta \, \sin(3x) \sin(2y+\beta).\label{E-es23-120yy}
\end{align}
\end{subequations}

For $(\eta,\beta) \in \set{(0,0);(\pi,0);(\pic,\pic)}$, we have
\begin{equation*}
\partial_{x^2}\Phi_{\beta,\theta}(x,\eta)\equiv 0, ~~\partial_{y^2}\Phi_{\beta,\theta}(x,\eta)\equiv 0\,.
\end{equation*}

The critical zero $(\xi,\eta)$ has order at least $3$, if
$\partial_y\Phi_{\beta,\theta}(\xi,\eta) = 0$ and $\partial_{xy}\Phi_{\beta,\theta}(\xi,\eta) = 0$. This system implies that
\begin{equation*}
\big( 3 \cos(3\xi) \sin(2\xi) - 2 \cos(2\xi) \sin(3\xi) \big) \cos(3\eta) \cos(2\eta+\beta) = 0\,.
\end{equation*}

The term in the parenthesis is equal to $2 \sin^3(\xi) (4 \cos^2(\xi) + 1)$, and hence does not vanish since $\xi \in (0,\pi)$. The second factor does not vanish either since $(\eta,\beta) \in \set{(0,0);(\pi,0);(\pic,\pic)}$. It follows that the second derivative $\partial_{xy}\Phi_{\beta,\theta}$ does not vanish at $(\xi,\eta)$. The proof of Assertion~(2) is complete.\medskip

\emph{Assertion~(3).}~ Let $\xi \in \set{0,\pi}$ and $\varepsilon = \cos(\xi)$. The only derivatives of $\Phi$,  of order less than or equal to $4$, which are not identically identically zero on $\set{x=\xi}$ are,
\begin{subequations}\label{E-es23-130}
\begin{align}
\partial_x\Phi(\xi,y) &= 2 \cos\theta \sin(3y) + 3 \varepsilon\, \sin\theta \sin(2y+\beta),\label{E-es23-13x}\\[5pt]
\partial_{xy}\Phi(\xi,y) &= 6 \cos\theta \cos(3y) + 6 \varepsilon\, \sin\theta \cos(2y+\beta),\label{E-es23-13xy}\\[5pt]
\partial_{x^3}\Phi(\xi,y) &= -8 \cos\theta \sin(3y) - 27 \varepsilon\, \sin\theta \sin(2y+\beta),\label{E-es23-13x3}\\[5pt]
\partial_{xy^2}\Phi(\xi,y) &= -18 \cos\theta \sin(3y) - 12 \varepsilon\, \sin\theta \sin(2y+\beta),\label{E-es23-13xy2}\\[5pt]
\partial_{x^3y}\Phi(\xi,y) &= -24 \cos\theta \cos(3y) -54 \varepsilon\, \sin\theta \cos(2y+\beta),\label{E-es23-13x3y}\\[5pt]
\partial_{xy^3}\Phi(\xi,y) &= -54 \cos\theta \cos(3y) - 24 \varepsilon\, \sin\theta \cos(2y+\beta).\label{E-es23-13xy3}
\end{align}%
\end{subequations}

The point $(\xi,\eta)$ is a boundary critical zero if and only if $\partial_{x}\Phi(\xi,\eta)=0$. This occurs either for $(\eta,\beta) \in \set{(0,0);(\pi,0);(\pic,\pic)}$, or for $\eta \not \in \set{0,\pic,\pit,\pi}$, and $\theta$ given by $\cot\theta = - \varepsilon f(\beta,\eta)$.\smallskip

The critical point $(\xi,\eta)$ has order at least $3$ if and only if $\partial_{xy}\Phi(\xi,\eta)=0$, i.e., $g(\beta,\eta)=0$.\smallskip

The boundary critical zero $(\xi,\eta)$ has order at least $4$ if and only if $\partial_{x^3}\Phi(\xi,\eta)=0$ and $\partial_{xy^2}\Phi(\xi,\eta)=0$ simultaneously. For this to occur, we must have $\sin(3\eta) \sin(2\eta+ \beta) = 0$. The relations $\partial_{x^3y}\Phi(\xi,\eta)=0$ and $\partial_{xy^3}\Phi(\xi,\eta)=0$ can only occur simultaneously if $\cos(3\eta) \cos(2\eta+\beta) = 0$. According to the previous relation, this means that a boundary critical zero has order less than or equal to $4$.
\hfill \qed

\begin{remark}\label{R-es23-2}
We shall refine the analysis of Assertion~(3) later on, in Subsection~\ref{SS-sc23-6}.
\end{remark}%

\begin{remark}\label{R-es23-4}
The function $f(\beta,y)$, defined in \eqref{E-es23-108f}, and the function $g(\beta,y)$, defined in \eqref{E-es23-108g}, are crucial to understand the structure of the nodal sets. We study them in Subsection~\ref{SS-es23-f-g}.
\end{remark}%

\begin{proposition}\label{P-es23-6}
Define the set
\begin{equation}\label{E-es23-140z}
\cZ_{\beta} :=\cZ (\Phi_{\beta,0}) \cap \cZ(\Phi_{\beta,\pib}).
\end{equation}
This is the set of common zeros of the eigenfunctions $\Phi_{\beta,\theta}$, where $\beta$ is fixed, and $\theta$ varies in $(0,\pib)$. Define the function
\begin{equation}\label{E-es23-140p}
P_{\beta}(x,y) = \sin(2x) \sin(3y) \, \sin(3x) \sin(2y+\beta).
\end{equation}
\begin{enumerate}
  \item For $\beta \in \set{0,\pic}$ and $\theta \in (0,\pib)$, the set $\cZ_{\beta}$ consists in the line $\set{y=\beta}$, and a finite number of regular points for $\Phi_{\beta,\theta}$ on $M_1\sm \set{y=\beta}$.
  \item For $\beta \in (0,\pic)$ and $\theta \in (0,\pib)$, the set $\cZ_{\beta}$ consists in finitely many regular points for $\Phi_{\beta,\theta}$.
  \item For $\beta \in (0,\pic)$, the nodal set $\cZ(\Phi_{\beta,\theta})$ does not meet the set $\set{P_{\beta} > 0}$, and only meets the nodal sets $\cZ(\Phi_{\beta,0})$ and $\cZ(\Phi_{\beta,\pib})$ at the points in $\cZ_{\beta}$.
  \item For $\beta \in \set{0,\pic}$, the nodal set $\cZ(\Phi_{\beta,\theta})$ does not meet the set $\set{P_{\beta} > 0}$, and only meets the nodal sets $\cZ(\Phi_{\beta,0})$ and $\cZ(\Phi_{\beta,\pib})$ at the points in $\cZ_{\beta}$, or along the common line $\set{y=\beta}$.
\end{enumerate}
\end{proposition}%

This proposition, whose proof is clear, is best understood by looking at Figure~\ref{F-es23-ex-checker} (see also Figures~\ref{F-es23-gc-inf} and \ref{F-es23-sc4-inf}).\medskip

\begin{figure}[h]
\centering
\begin{subfigure}[t]{.35\textwidth}
\centering
\includegraphics[width=\linewidth]{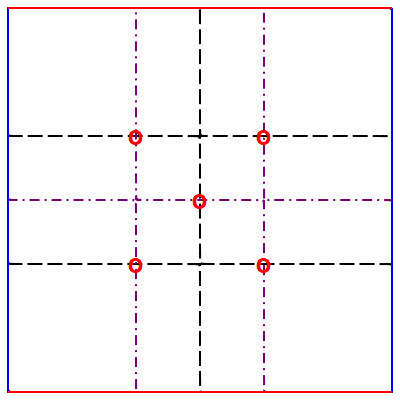}
\caption{Common zeros}
\end{subfigure}
\begin{subfigure}[t]{.36\textwidth}
\centering
\includegraphics[width=\linewidth]{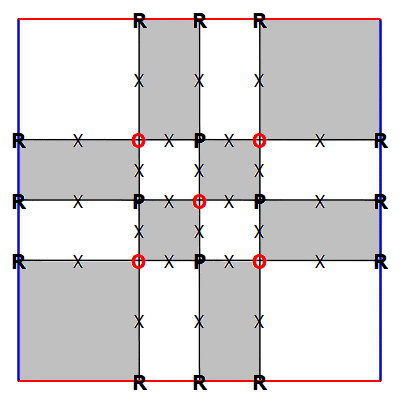}
\caption{Prohibitions}
\end{subfigure}
\caption{The checkerboards for $\beta=0$ and $\theta \in (0,\pib)$}\label{F-es23-ex-checker}
\end{figure}

In Figure~\ref{F-es23-ex-checker}(A), the nodal lines of $\Phi_{\beta,0}$ appear as dashed black lines;  the nodal lines of $\Phi_{\beta,\pib}$ appear as dot/dashed purple lines; the lines $\set{y=0}$ and $\set{y=\pi}$ are common to both functions and appear in red. The points of $\cZ_{\beta}\sm \set{y=0}$ are marked ``O'' (in red). The ``checkerboard''  in Figure~(B) displays the prohibitions: the nodal set $\cZ(\Phi_{\beta,\theta})$ cannot visit the domains in gray, and cannot cross the segments marked with ``X''. The ``P'' denote points through which the nodal set cannot pass. The points marked ``R'' denote regular points (in particular, the nodal set cannot hit the boundary at regular points). We will use similar checkerboards in Subsections~\ref{SS-sc23-6} and \ref{SS-sc23-34}, to describe the nodal sets, using Propositions~\ref{P-es23-2}, \ref{P-es23-4}, \ref{P-es23-6} and Lemma~\ref{L-es23-4}

\begin{notation}\label{N-es23-2}
Given a bounded domain $\Omega$, let $\delta(\Omega)$ denote its smallest Dirichlet eigenvalue.
\end{notation}%

Let $\Phi \in \cE_{[2,3]} = \cE_{\lambda_7}$ be any eigenfunction, and let $\Omega$ be any nodal domain of  $\Phi$. Then, it is well-known that $\delta(\Omega) = \lambdah(2,3) = 13$. Clearly, any nodal domain of the function $P_{\beta}$ is an open set contained in some nodal domain of $\Phi_{\beta,0}$ or $\Phi_{\beta,\pib}$.

\begin{lemma}\label{L-es23-4}
For $\beta \in [0,\pic]$ and $\theta \in (0,\pib)$, no nodal domain $\Omega_P$ of the function $P_{\beta}$ can contain a closed nodal line of $\Phi_{\beta,\theta}$, or a nodal line whose end points are located on $\Omega \cap \partial M_1$.
\end{lemma}%

\pf Indeed, such a curve would bound a nodal domain $\omega$ of $\Phi_{\beta,\theta}$, strictly contained in $\Omega_P$. On the one hand, by domain monotonicity of Dirichlet eigenvalues, we would have $\delta(\omega) > \delta(\Omega_P) \ge \lambdah(2,3)$ (the second inequality follows from the fact that the nodal domains of $P_{\beta}$ are contained in nodal domains of the functions $\Phi_{\beta,0}$ or $\Phi_{\beta,\pib}$).  On the other hand, we have $\delta(\omega) = \lambdah(2,3)$, a contradiction. \hfill \qed

\subsection{Analysis of the functions $f(\beta,y)$ and $g(\beta,y)$}\label{SS-es23-f-g}

Recall the expression of the function $f(\beta,y)$ defined in \eqref{E-es23-108f}. For $y \in (0,\pic) \cup (\pic,\pit) \cup (\pit,\pi)$,
\begin{equation*}
f(\beta,y) := \frac 32\, \frac{\sin(2y+\beta)}{\sin(3y)}.
\end{equation*}

We have,
\begin{equation*}
\partial_y f(\beta,y) = \frac{3}{2\sin^2(3y)}\, g(\beta,y),
\end{equation*}
where
\begin{equation*}
g(\beta,y) = 2\sin(3y) \, \cos(2y+\beta) - 3 \cos(3y) \, \sin(2y+\beta).
\end{equation*}
was defined in \eqref{E-es23-108g}. Expanding $g(\beta,y)$, and using the identity $\frac{1}{\sin^2(y)}= 1+\cot^2(y)$, we obtain
\begin{equation*}
g(\beta,y) = \sin^5(y) \left\lbrace  2 \cos\beta ( 1+5\cot^2(y)) - \sin\beta \cot(y) ( 5 + 3 \cot^4(y)) \right\rbrace.
\end{equation*}

It follows that
\begin{equation}\label{E-es23-206}
g(\beta,y) = 0 \text{~and~}y \in (0,\pi) \Leftrightarrow \cot\beta = \frac{1}{2}\,
\frac{\cot(y)\left( 5 + 3 \cot^4(y) \right)}{1+5\cot^2(y)}.
\end{equation}

\begin{lemma}\label{L-es23-202}
The function,
\begin{equation}\label{E-es23-208}
h(t) := \frac{1}{2}\, \frac{t(5+3t^4)}{1+5t^2},
\end{equation}
is an increasing bijection from $\R$ to $\R$. As a consequence,
the function
\begin{equation}\label{E-es23-210}
y \mapsto  \arccot \left(h(\cot y)\right)
\end{equation}
is an increasing bijection from $(0,\pi)$ to $(0,\pi)$. Furthermore,
\begin{equation}\label{E-es23-210a}
\begin{array}{|c|c|c|c|c|c|c|c|}
\hline &&&&&&&\\
t&-\infty &-\frac{1}{\sqrt{3}} & 0 & \frac{1}{\sqrt{3}} & 1 &\sqrt{3}&\infty  \\[5pt]
\hline &&&&&&&\\
h(t)&-\infty &-\frac{1}{\sqrt{3}}&0&\frac{1}{\sqrt{3}}&\frac 23&  \sqrt{3}&\infty\\[5pt]
\hline \hline&&&&&&&\\
y &0 &\frac{\pi}{6}&\pid &\pic &\pib &\pit &\pi\\[5pt]
\hline &&&&&&&\\
\arccot \left(h(\cot y)\right)&0 &\frac{\pi}{6}&\arccot(\frac 23)&\pic&\pib&\pit&\pi\\[5pt]
\hline
\end{array}%
\end{equation}
\end{lemma}%

\pf A straightforward computation gives
\begin{equation*}
h'(t) = \frac{5}{2}\, \frac{(3t^2-1)^2(t^2+1)}{(1+5t^2)^2}\,.
\end{equation*}
The lemma follows. \hfill\qed\medskip

As a consequence, given $\beta \in (0,\pic)$, there exists a unique $y_{\beta} \in (0,\pic)$, such that $\beta = \arccot \left(h(\cot y_{\beta})\right)$. It follows that, given  $\beta \in (0,\pic)$, the function $y \mapsto f(\beta,y)$ varies as indicated \eqref{E-es23-T1}.

\begin{equation}\label{E-es23-T1}
\begin{array}{c|cccccccccccccc}
y & 0 & \vo{3} & y_{\beta} & \vo{5} & \vo{6} & \pic & \vo{8} & \vo{9} &&\pit& \vo{8} & \vo{9} & ~~\pi\\[5pt]
\hline\\
\partial_y f(\beta,y) & \vo{2} & - & 0 & + & \vo{6} & \| & \vo{8} & + & \vo{10} &\| &&+&\\[5pt]
\hline\\
 & \infty & & & & \infty & \| &  & & \infty &\|&&& \infty\\[5pt]
f(\beta,y) &  & \searrow& & \nearrow  &  & \| &  &\nearrow  && \|&&\nearrow&&\\[5pt]
& & & m_{\beta} & & & \| & -\infty & &&\|&-\infty&&\\[5pt]
\hline
\end{array}%
\end{equation}
where
\begin{equation}\label{E-es23-20T1}
m_{\beta} := \frac 32 \, \frac{\sin(2 y_{\beta}+\beta)}{\sin(3y_{\beta})}\,.
\end{equation}

\begin{figure}[!ht]
\centering
\includegraphics[scale=0.4]{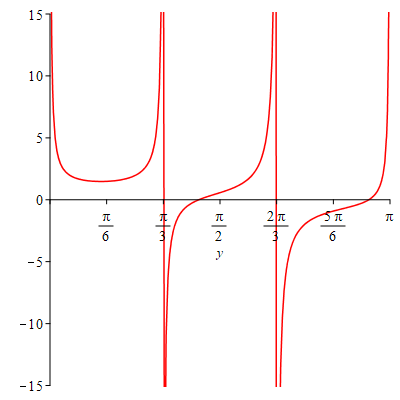}
\caption{Graph of the function $f(\beta,\cdot)$, for $\beta=\frac{\pi}{8}$}\label{F-23-graph-f}
\end{figure}

From Lemma~\ref{L-es23-202}, the definition of $g(\beta,y)$ in \eqref{E-es23-108g}, and the uniqueness of the zero of the function $y \mapsto g(\beta,y)$, we deduce that
\begin{equation}\label{E-es23-212}
\left\{
\begin{array}{l}
\beta = \pif \Rightarrow y_{\beta} = \pif\,,\\[5pt]
\beta > \pif \Rightarrow \beta > y_{\beta} > \pif\,,\\[5pt]
\beta < \pif \Rightarrow \beta < y_{\beta} < \pif\,,
\end{array}
\right.
\end{equation}
and from the definition of $g(\beta,y)$, see \eqref{E-es23-108g},
\begin{equation}\label{E-es23-212a}
\left\{
\begin{array}{l}
y_{\pic- \beta} = \pic - y_{\beta}\,,\\[5pt]
y_{\pif} = \pif\,.
\end{array}
\right.
\end{equation}

It turns out that $y$ and $\beta$ play a symmetric role. Instead of fixing $\beta$, we could fix some $\eta \in (0,\pic)$. Then, there exists a unique $\beta_{\eta} = \arccot(h(\cot \eta))$, and the function $y \mapsto f(\beta_{\eta},y)$, has an infimum $m_{\beta_{\eta}}$ at $\eta$. \medskip

The relation $\cot(\beta) = h\left( \cot(y) \right)$ can also be rewritten as,
\begin{equation}\label{E-es23-206T}
\tan(\beta) = \ell\left( \tan(y) \right),
\end{equation}
where
\begin{equation}\label{E-es23-206L}
\ell(t) := \frac{2 t^3 (5+t^2)}{3+5t^4}.
\end{equation}

We conclude that
\begin{equation}\label{E-es23-212b}
\beta_{\eta} \sim \frac{10}{3}\, \eta^3, \text{~as~} \eta \to 0.
\end{equation}

From \eqref{E-es23-20T1}, \eqref{E-es23-210} and \eqref{E-es23-212}, we conclude that the function $\beta \mapsto m_{\beta}$ satisfies,
\begin{equation}\label{E-es23-214}
\left\{
\begin{array}{l}
m_{\pic - \beta} = m_{\beta}\,,\\[5pt]
m_{\pif} = \frac 32\,.\\[5pt]
\end{array}%
\right.
\end{equation}

Writing $m_{\eta} =f(\beta_{\eta},\eta)$ and using \eqref{E-es23-212b}, we obtain that
\begin{equation}\label{E-es23-214b}
\lim_{\eta \to 0}m_{\eta} = 1.
\end{equation}

Finally from the definitions of $m_{\beta} = f(\beta,y_{\beta})$ and of $y_{\beta}$, and using \eqref{E-es23-212}, we conclude that
\begin{equation}\label{E-es23-216}
\left\{
\begin{array}{l}
m'_{\beta} > 0, \text{~for~} \beta < \pif\,,\\[5pt]
m'_{\beta} <  0, \text{~for~} \beta > \pif\,,\\[5pt]
1 \le m_{\beta} \le \frac 32\,.
\end{array}%
\right.
\end{equation}
Furthermore, $m_{\beta}=1$ if and only if $\beta = 0$ or $\beta = \pib$, and
$m_{\beta}=\frac 32$ if and only if $\beta = \pif$.\medskip

Define the value $\theta_{\beta}$ by the relation
\begin{equation}\label{E-es23-220}
\cot(\theta_{\beta}) = \frac 32 \, \frac{\sin(2y_{\beta}+\beta)}{\sin(3y_{\beta})} = \frac 32\, \frac{\sin(2y_{\beta}+\arccot(h(\cot y_{\beta})))}{\sin(3y_{\beta}))}\,.
\end{equation}

Equivalently, the value $\theta_{\beta}$ is defined by the equations
\begin{equation}\label{E-es23-222}
\left\{
\begin{array}{l}
\cot(\beta) = h\left( \cot y_{\beta} \right)\,,\\[5pt]
2 \cos\theta_{\beta} \, \sin(3y_{\beta}) - 3\, \sin\theta_{\beta} \, \sin(2y_{\beta}+\beta) = 0\,.
\end{array}
\right.
\end{equation}

From the definition of $f$ and \eqref{E-es23-214}, we conclude that
\begin{equation}\label{E-es23-224}
\left\{
\begin{array}{l}
\theta_{\pic - \beta} = \theta_{\beta}\,,\\[5pt]
\theta_{\pif} = \arccot(3/2)\,,\\[5pt]
\lim_{\beta \to 0}\theta_{\beta} = \lim_{\beta \to \pib}\theta_{\beta} = \pid\,.
\end{array}%
\right.
\end{equation}
The graph of $\beta \mapsto \theta_{\beta}$ is symmetric with respect to $\beta=\pif$, the function decreases from $\pid$ to $\arccot(3/2) \approx 0.588003$ as $\beta$ increases from $0$ to $\pif$, and then increases to $\pid$ when $\beta$ continues increasing to $\pic$.\medskip

\begin{remark}\label{R-es23-202}
As pointed out above, when looking at the set $\set{g(\beta,y)=0}$, $\beta$ and $y$ play symmetric roles. Given $\beta$, there is a unique zero $y_{\beta}$, and conversely, given $y$, there is a unique zero $\beta_{y}$. In order to visualize the functions $m_{\beta}$ and $\theta_{\beta}$, it is simpler to look at $\beta$ as a function of $y$, and to visualize $m_{\beta_y}$ and $\theta_{\beta_y}$ as functions of $y$. Figure~\ref{F-fig23-ymt} displays the graphs of the functions $\beta_{y}=\arccot \left(h(\cot y)\right)$, $m_{\beta_{y}}$ and $\theta_{\beta_{y}}$.
\end{remark}%

\begin{figure}[ht]
\centering
\begin{subfigure}[t]{.30\textwidth}
\centering
\includegraphics[width=\linewidth]{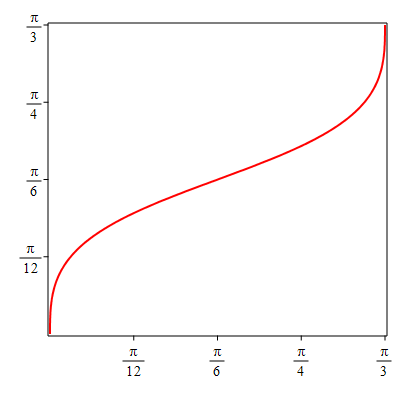}
\caption{$\beta_y = \arccot \left(h(\cot y)\right)$}
\end{subfigure}
\begin{subfigure}[t]{.30\textwidth}
\centering
\includegraphics[width=\linewidth]{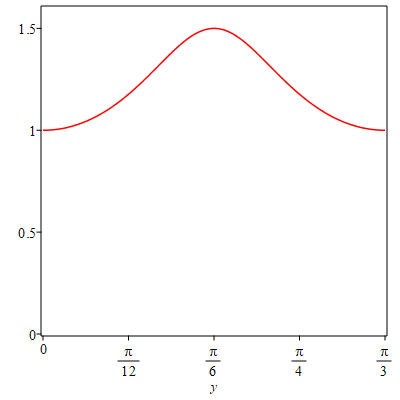}
\caption{$m_{\beta_y}$}
\end{subfigure}
\begin{subfigure}[t]{.30\textwidth}
\centering
\includegraphics[width=\linewidth]{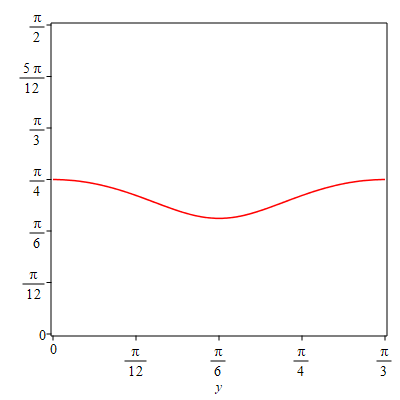}
\caption{$\theta_{\beta_y}$}
\end{subfigure}
\caption{Graphs as functions of $y$}\label{F-fig23-ymt}
\end{figure}

\FloatBarrier

\subsection{The general case $\beta \in (0,\pic)$ and $\theta \in (0,\pib)$}\label{SS-sc23-6}

We write the eigenfunction as
\begin{equation}\label{E-es23gc-2}
\begin{array}{ll}
\Phi_{\beta,\theta}(x,y)&= \cos\theta \sin(2x) \sin(3y) + \sin\theta \sin(3x) \sin(2y+\beta),\\[5pt]
&= \sin(x) \, \Psi_{\beta,\theta}(x,y), \text{~where},\\[5pt]
\Psi_{\beta,\theta}(x,y)&= 2\cos\theta \cos(x) \sin(3y) \\[5pt]
&~~~~+ \sin\theta (4\cos^2(x)-1) \sin(2y+\beta),
\end{array}%
\end{equation}
and we also use the abbreviation $\Phi$ or $\Psi$ to simplify notation.\medskip

The checkerboards associated with the function $P_{\beta}$,
\begin{equation*}
P_{\beta}(x,y) = \sin(2x) \sin(3y) \, \sin(3x) \sin(2y+\beta)\,,
\end{equation*}
are given in  Figure~\ref{F-es23-gc-inf}. According to Proposition~\ref{P-es23-6},
the nodal set $\cZ(\Phi)$ is contained in the set $\set{P_{\beta}\le 0}$ (the rectangles in white). The points marked ``O'' in the figure are the common zeros of the family of functions $\Phi_{\beta,\theta}$, for a fixed $\beta \in (0,\pic)$, and all $\theta \in (0,\pib)$. The function $\Phi$ does not vanish at the points marked ``P'', so that $\cZ(\Phi)$ does not pass through these points. The points marked ``R'' are regular points of the boundary, the partial derivative $\partial_x\Phi$ does not vanish at these points.

\begin{figure}[h]
\centering
\begin{subfigure}[t]{.35\textwidth}
\centering
\includegraphics[width=\linewidth]{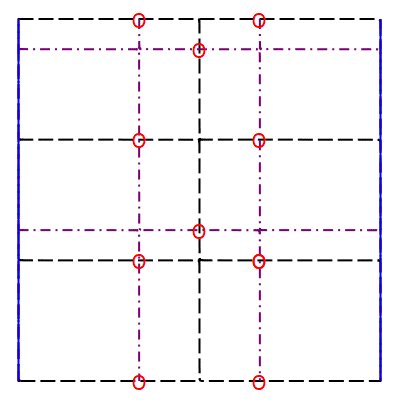}
\caption{Grid with common zeros}
\end{subfigure}
\begin{subfigure}[t]{.352\textwidth}
\centering
\includegraphics[width=\linewidth]{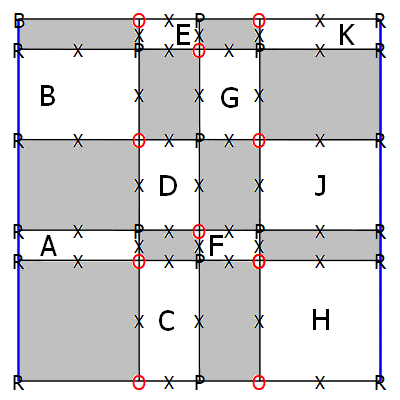}
\caption{Grid with prohibitions}
\end{subfigure}
\caption{Localization of the nodal set $\cZ(\Phi_{\beta,\theta})$}\label{F-es23-gc-inf}
\end{figure}

\FloatBarrier

\begin{properties}\label{Ps-es23gc-2}The following properties hold.
\begin{enumerate}
  \item For $y_0 \in \set{0,\pic, \pit, \pi}$, the function $\Psi(x,y_0)$ only vanishes at the points $x=\pic$ and $x=\pit$. The nodal set $\cZ(\Phi)$ cannot hit the corresponding horizontal lines elsewhere.
  \item For $y_0 \in \set{\pib - \frac{\beta}{2}, \pi-\frac{\beta}{2}}$,  the function $\Psi(x,y_0)$ only vanishes at the point $x=\pib$. The nodal set $\cZ(\Phi)$ cannot hit the corresponding horizontal lines elsewhere.
  \item By \eqref{E-es23-110x} and \eqref{E-es23-110y}, at the common zeros marked ``O'', the tangent to the nodal line is oriented SouthEast--NorthWest, pointing to the white domains, except on the bottom line where the orientation is SouthWest--NorthEast.
  \item According to Subsection~\ref{SS-es23-f-g} and \eqref{E-es23-T1}, the function $\Psi(0,y)$ has precisely one zero in each interval $(\pic,\pib-\frac{\beta}{2})$ and $(\pit,\pi-\frac{\beta}{2})$. The nodal set $\cZ(\Phi)$ hits the boundary at a critical zero of order $2$.
  \item According to Subsection~\ref{SS-es23-f-g} and \eqref{E-es23-T1}, the function $\Psi(\pi,y)$ has precisely one zero in each interval  $(\pib-\frac{\beta}{2},\pit)$ and $(\pi-\frac{\beta}{2},\pi)$.  The nodal set $\cZ(\Phi)$ hits the boundary at a critical zero of order $2$. In the interval $(0,\pic)$, the function $\Psi(\pi,y)$ has 0, 1, or 2 zeros, depending on whether $\cot\theta$ is less than, equal to, or larger than $m_{\beta}$ (defined in Equation~\eqref{E-es23-20T1}). In the first case, the nodal set does not hit the boundary. In the second case, it hits the boundary at a critical zero of order $3$. In the third case, it hits the boundary at two critical zeros of order $2$.
\end{enumerate}
\end{properties}%

From these properties, we deduce the general aspect of the nodal set $\cZ(\Phi)$, with reference to the right hand picture in Figure~\ref{F-es23-gc-inf}.

\begin{description}
  \item[Column~1] The nodal set in the rectangles A and B is a simple curve with end points a boundary critical zero of order $2$ and the common zero marked ``O''.
  \item[Column~2] The nodal set in the rectangles C, D and E is a simple curve with end points the common zeros marked ``O''.
  \item[Column~3] The nodal set in the rectangles F and G is a simple curve with end points the common zeros marked ``O''.
  \item[Column~4] In the rectangles J and K, the nodal set is a simple curve with end points the common zero marked ``O'' and a boundary critical zero of order $2$.
      In the rectangle H, there are three different cases.
      \begin{description}
        \item[$0 < \theta  < \theta _{\beta}$] The nodal set is a simple curve joining the two common zeros marked ``O''.
        \item[$\theta = \theta _{\beta}$~~~~~] The nodal set consists in two arcs entering the rectangle at the common zeros marked ``O'', and meeting at the boundary critical zero $(\pi,y_{\beta})$, of order $3$.
        \item[$\theta _{\beta} < \theta < \pib$] The nodal set consists in two disjoint arcs entering the rectangle at the common zeros marked ``O'', and hitting the boundary at two distinct critical zeros, each of order $2$.
      \end{description}
\end{description}

\begin{remark}\label{R-es23gc-4}
One can localize the position of the nodal curves by looking at their intersections with the lines $\set{x=x_0}$, i.e., at the zeros of the function $\Psi(x_0,y)$, and by making use of \eqref{E-es23-T1}.
\end{remark}%

One can summarize the preceding discussion as follows, see Figure~\ref{F-es23-gc}. Fix some $\beta \in (0,\pic)$. Let $y_{\beta} \in (0,\pic)$ and  $\theta_{\beta} \in (0,\pib)$, be defined by \eqref{E-es23-206} and \eqref{E-es23-220} respectively.\smallskip

The nodal sets of the family $\{\Phi_{\beta,\theta}\}$, with $\beta$ fixed and $\theta \in (0,\pib)$ present a bifurcation at $\theta = \theta_{\beta}$, as illustrated by Figure~\ref{F-es23-gc} (in this figure, $\beta = \pid$). \smallskip

When $\theta = \theta_{\beta}$, Figure~(B), the nodal set consists in three simple regular curves which do not intersect in $M_1$. Two of these curves hit the boundary at the critical zero $(\pi,y_{\beta})$, of order three. One of them hits the boundary part $\set{x=0}$; the other hits the boundary part $\set{x=\pi}\times (\pic,\pit)$. The two end points of the third curve belong to $\set{x=0}$ and $\set{x=\pi}$ respectively. Except for the afore mentioned boundary critical zero of order $3$, the boundary critical zeros have order $2$. There is no interior critical zero.\smallskip

When $0 < \theta < \theta_{\beta}$, Figure~(A), the nodal set consists in three non-intersecting simple regular curves. Two of them have one end point on each part of the boundary. The third has two distinct end points on $\set{x=\pi}$. There are no interior critical zero. The boundary critical zeros all have order $2$.\smallskip

When $\theta_{\beta} < \theta < \pib$, Figure~(C), the nodal set consists in two non-intersecting simple regular curves. They both have one end point on $\set{x=0}$, the other on $\set{x=\pi}$. These end points are critical zeros of order $2$. There are no interior critical zeros. \smallskip

\emph{In all three cases, the number of nodal domains (on $M_1$) is $3$.} \medskip

\newcommand{\twc}{0.75}
\begin{figure}[ht]
\centering
\begin{subfigure}[t]{\twc\textwidth}
\centering
\includegraphics[width=\linewidth]{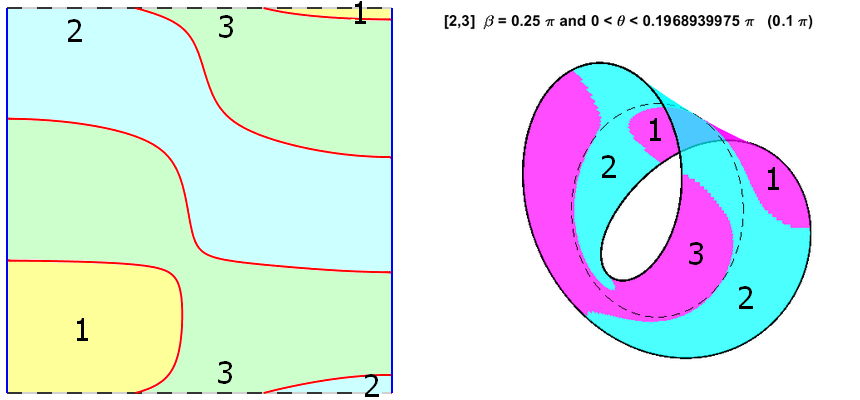}
\caption{$0 < \theta < \theta_{\beta}$}\vspace{2mm}
\end{subfigure}
\begin{subfigure}[t]{\twc\textwidth}
\centering
\includegraphics[width=\linewidth]{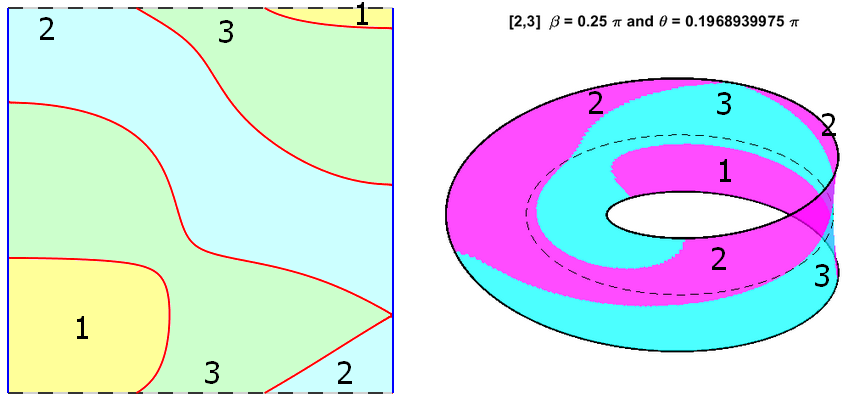}
\caption{$\theta = \theta_{\beta}$}\vspace{2mm}
\end{subfigure}
\begin{subfigure}[t]{\twc\textwidth}
\centering
\includegraphics[width=\linewidth]{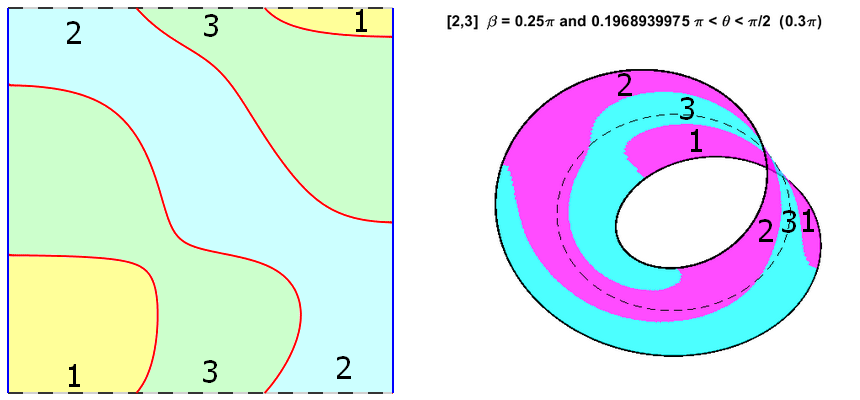}
\caption{$\theta_{\beta} < \theta < \pib$}
\end{subfigure}
 \caption{$\cZ(\Phi_{\beta,\theta})$, with $0 < \beta < \pic$ fixed,
 and $0 < \theta < \pib$}\label{F-es23-gc}
\end{figure}

\FloatBarrier

\subsection{Special values $\beta=0$ and $\beta=\pic$, with $\theta \in (0,\pib)$}\label{SS-sc23-34}

The nodal sets $\cZ(\Phi_{\pic,\theta})$ are nicer to visualize. They are displayed in Figure~\ref{F-es23-sc4}. The functions $\Phi_{0,\theta}$ are Dirichlet eigenfunctions of the square $(0,\pi)^2$. Their nodal sets have already been studied in \cite{Pl}, see also \cite{BH} for more details. They are given in Figure~\ref{F-es23-sc3}. Recall that $\cZ(\Phi_{\pic,\theta}) = \tau_{\pic}\, \cZ(\Phi_{0,\theta})$, see \eqref{E-esmn-es23} and Figure~\ref{F-23-tau-sym}.

\newcommand{\twd}{0.75}
\begin{figure}[!h]
\centering
\begin{subfigure}[t]{\twc\textwidth}
\centering
\includegraphics[width=\linewidth]{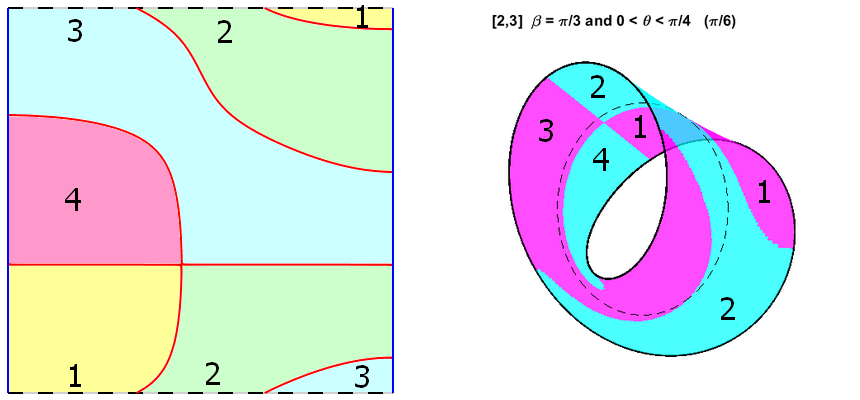}
\caption{$0 < \theta < \pid$}\vspace{2mm}
\end{subfigure}
\begin{subfigure}[t]{\twc\textwidth}
\centering
\includegraphics[width=\linewidth]{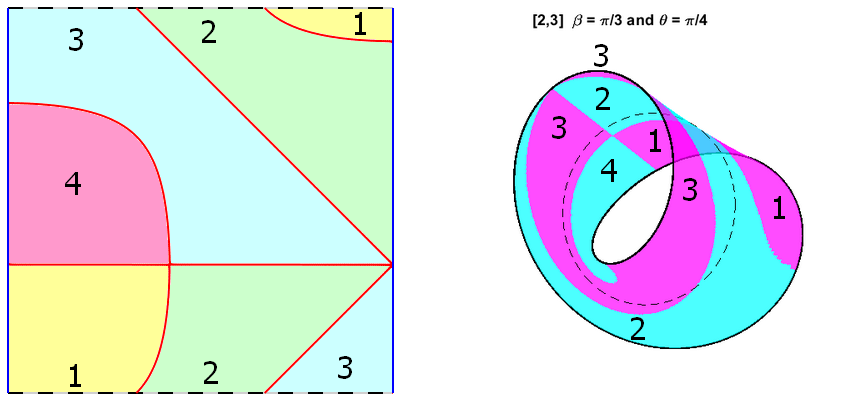}
\caption{$\theta = \pid$}\vspace{2mm}
\end{subfigure}
\begin{subfigure}[t]{\twc\textwidth}
\centering
\includegraphics[width=\linewidth]{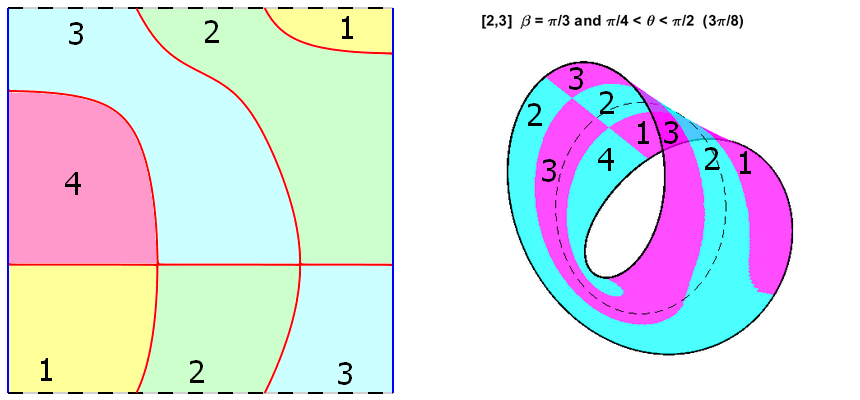}
\caption{$\pid < \theta < \pib$}
\end{subfigure}
 \caption{$\cZ(\Phi_{\pic,\theta})$ for $0 < \theta < \pib$}
 \label{F-es23-sc4}
\end{figure}
\FloatBarrier

\begin{figure}[!h]
\centering
\begin{subfigure}[t]{.30\textwidth}
\centering
\includegraphics[width=\linewidth]{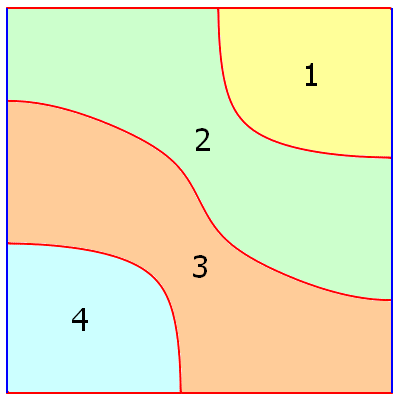}
\caption{$0 < \theta < \pid$}
\end{subfigure}
\begin{subfigure}[t]{.30\textwidth}
\centering
\includegraphics[width=\linewidth]{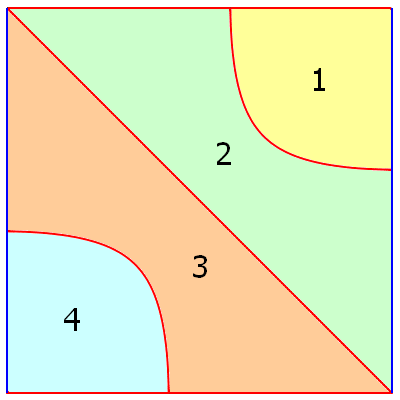}
\caption{$\theta = \pid$}
\end{subfigure}
\begin{subfigure}[t]{.30\textwidth}
\centering
\includegraphics[width=\linewidth]{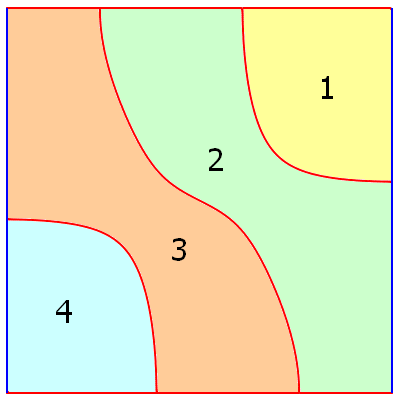}
\caption{$\pid < \theta < \pib$}
\end{subfigure}
\caption{$\cZ(\Phi_{0,\theta})$ for $0 < \theta < \pib$}
\label{F-es23-sc3}
\end{figure}

We summarize the analysis of the nodal sets for $\beta = \pic$. Write
\begin{equation}\label{E-es23-sc4a}
\left\{
\begin{array}{ll}
\Phi_{\pic,\theta}(x,y) & = \cos\theta\, \sin(2x) \, \sin(3y) + \sin\theta \, \sin(3x)\, \sin(2y+\pic)\\[5pt]
& = - 2\, \sin(x) \, \sin(y-\pic) \, \Gamma_{\theta}(x,y-\pic)\,,
\end{array}
\right.
\end{equation}
where
\begin{equation}\label{E-es23-sc4b}
\begin{array}{ll}
\Gamma_{\theta}(x,y-\pic) &= \cos\theta \cos(x) (4\cos^2(y-\pic)-1)\\[5pt]
 &~~~+ \sin\theta (4\cos^2(x)-1) \cos(y-\pic)\,.
\end{array}%
\end{equation}

Consider the checkerboards associated with the function,
\begin{equation}\label{E-es23-sc4c}
\begin{array}{ll}
P_{\pic}(x,y) & = \sin(2x) \sin(3y) \, \sin(3x) \sin(2y+\pic),\\[5pt]
& = \sin(2x) \sin(3(y-\pic)) \sin(3x) \sin(2(y-\pic)),
\end{array}%
\end{equation}
and apply Proposition~\ref{P-es23-6}: the nodal set
$\Phi_{\pic,\theta}$ can only visit the white domains. It passes through the points marked ``O'', does not pass through the points marked ``P''. The points marked ``R'' are regular points at the boundary. The corresponding checkerboard appears in Figure~\ref{F-es23-sc4-inf}.

\begin{figure}[h]
\centering
\begin{subfigure}[t]{.35\textwidth}
\centering
\includegraphics[width=\linewidth]{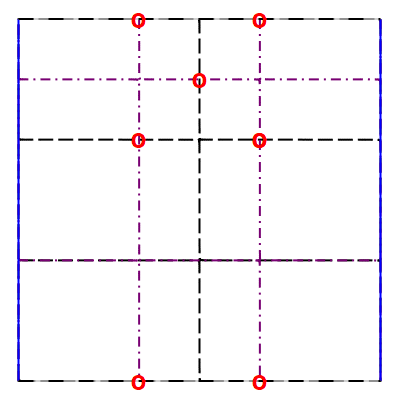}
\caption{Grid with common zeros}
\end{subfigure}
\begin{subfigure}[t]{.35\textwidth}
\centering
\includegraphics[width=\linewidth]{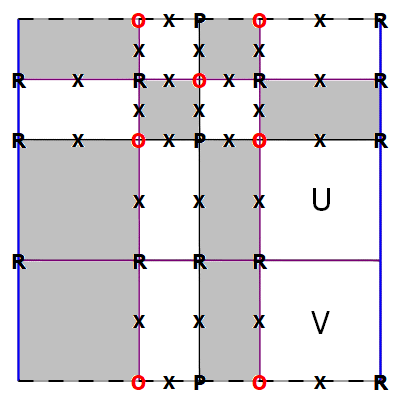}
\caption{Grid with prohibitions}
\end{subfigure}
\caption{Location of the nodal set $\cZ(\Phi_{\pic,\theta})$}
\label{F-es23-sc4-inf}
\end{figure}

Propositions~\ref{P-es23-2}--\ref{P-es23-6} and Lemma~\ref{L-es23-4} determine the nodal set $\cZ(\Phi_{0,\theta})$ except in the squares marked ``U'' and ``V'', where a further analysis is necessary.\medskip

Analyzing the functions $\Gamma_{\theta}(x,0)$ and $\Gamma_{\theta}(\pi,y-\pic)$, one finds that there are three cases:
\begin{enumerate}
  \item when $0 < \theta < \pid$, the nodal line entering ``U'' at $(\pit,\pit)$ does not hit $\set{y=\pic}$, and hits the boundary in $\set{\pi}\times (\pic,\pit)$ and there is a similar nodal line entering ``V'' at $(\pit,0)$;
  \item when $\pid < \theta < \pib$, the nodal line entering ``U'' at $(\pit,\pit)$ crosses $\set{y=\pic}$, does not hit the boundary $\set{\pi}\times (0,\pit)$, and exits $V$ at $(\pit,0)$;
  \item when $\theta = \pid$, the nodal set in $U\cup V$ consists in two line segments, one from $(\pit,\pit)$ to $(\pi,\pic)$, and a symmetric one from $(\pit,0)$ to $(\pi,\pic)$.
\end{enumerate}

In the first two cases, the critical zeros have order $2$; in the third case, the critical zero $(\pi,\pic)$ has order $4$, as indicated in Proposition~\ref{P-es23-4}, Assertion~(3).\medskip

\emph{In the three cases, the eigenfunction $\Phi$ has four nodal domains.}\medskip

Finally, Figure~\ref{F-23-tau-sym} illustrates the relations \eqref{E-esmn-es23}. More precisely, one can write
\begin{equation}\label{E-ef23-tau-sym}
\Phi_{\beta+\pic,\theta}(x,y) =
\left\{
\begin{array}{ll}
-\Phi_{\beta,\theta}(x,y-\pic) &\text{if~} y \in [\pic,\pi]\,,\\[5pt]
-\Phi_{\beta,\theta}(\pi-x,y+\pit) &\text{if~} y \in [0,\pic]\,.
\end{array}
\right.
\end{equation}

To pass from $\cZ(\Phi_{0,\pid})$ (Figure~\ref{F-23-tau-sym}, left) to $\cZ(\Phi_{\pic,\pid}) = \tau_{\pic}\,\cZ(\Phi_{0,\pid})$ (Figure~\ref{F-23-tau-sym}, right), divide the figure on the left along the black dashed horizontal line; translate the lower part upwards by $\pic$; translate the upper part downwards by $\pit$ and apply the symmetry with respect to $\set{x=\pib}$; glue the resulting domains along the horizontal red line. \medskip

The proof of Lemma~\ref{L-es23-2} is complete.\medskip

\begin{figure}[!ht]
\centering
\includegraphics[scale=0.35]{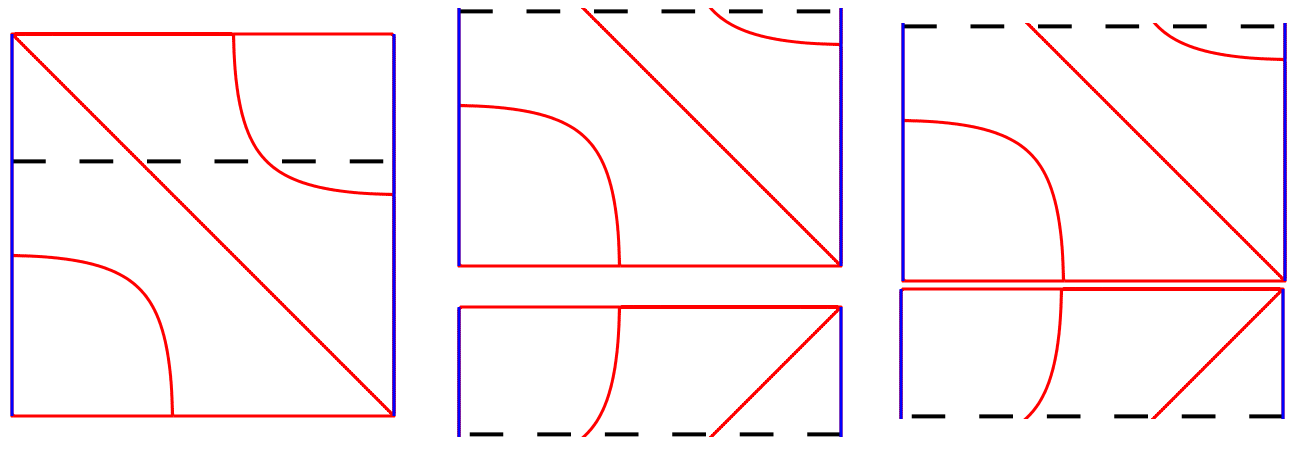}
\caption{From $\cZ(\Phi_{0,\pid})$ to $\cZ(\Phi_{\pic,\pid})$}\label{F-23-tau-sym}
\end{figure}

\FloatBarrier

\section{Isoperimetric Inequality and Faber-Krahn Property}\label{S-isop}

We follow the proof given by C. L\'ena in \cite{Le}, who refers to \cite{BM} and to the older \cite{Pe}.  A key role is played by the isoperimetric inequality in connection with the Faber-Krahn inequality. For this purpose, we use Howards's isoperimetric inequality for the Klein bottle, \cite[Section~7]{HHM}.

\subsection{Isoperimetric inequality}

\begin{theorem}\cite[Theorem 7, case 1]{HHM}
Let $S$ be a flat torus or a Klein bottle with shortest closed geodesic of length $\ell  $ and area $\cA(S)$. Given $0<A<\cA(S)$, the least-perimeter region of area $A$ is a circular disk if $0<A\leq \ell^2/\pi$.
\end{theorem}

\begin{corollary}
 Let $\Omega$ be an open domain of $M_1$ with perimeter $\ell (\Omega)$ and area $\cA(\Omega)$. If $\cA(\Omega)<\pi$ then $\ell (\Omega)$ is greater than or equal to the length of a circular disk of the same area:
$$\ell^2(\Omega)\geq 4\pi^2 \cA(\Omega).$$
\end{corollary}
\begin{proof}
We note indeed that an open set in $M_1$ can be considered as an open set of a flat Klein bottle. We have just to delete one line $x ={\rm const}$  on the Klein Bottle.
\end{proof}

\subsection{Faber-Krahn inequality}

\begin{proposition}\label{Faber}
If $\Omega$ is an open connected set of $M_1$  with a piecewise $C^1$ boundary, and such that $\cA(\Omega)\leq \pi$, then:
$$\lambda_1(\Omega)\, \cA(\Omega)\geq \pi j_{0,1}^2\,.$$
\end{proposition}
\begin{proof}
The proof is given in \cite{Pe} (with a more detailed proof analyzing more carefully the regularity assumptions given in \cite{BM}). In the two papers additional assumptions are made which are only
 used in the proof of the isoperimetric inequalities. We have replaced them by  Howards's result  \cite{HHM}. Our proof gives also a variant of L\'ena's result on the torus (note that L\'ena gives an alternative proof).
\end{proof}

\subsection{Weyl formula with control of the remainder}

Following the classical proof in Pleijel's foundational paper \cite{Pl},  we need to find an explicit lower bound for the counting function,
\begin{equation*}
N(\lambda)=\#\{k ~|~ \lambda_k < \lambda\},
\end{equation*}
whose main term as $\lambda \rightarrow +\infty$ is given by the classical Weyl formula,
\begin{equation}
N(\lambda) \sim \frac {\pi} 4 \,\lambda \,.
\end{equation}
The coefficient $\frac \pi 4 $ in front of $\lambda$  comes from the computation of $ \frac{1}{4 \pi} \cA(M_1)$ in 2D, where $\cA(M_1)$ is the area of $M_1$. In our case, we have  $\cA(M_1)=\pi^2$.\smallskip

This asymptotics is sufficient for showing that the number of Courant-sharp eigenvalues is finite. To actually determine the Courant-sharp eigenvalues, we need a lower bound for $N(\lambda)$, valid for any $\lambda \geq 9 $. The case of the square was treated in \cite{Pl}. For the M\"{o}bius strip, we prove the following lower bound.

\begin{proposition} The counting function of the Dirichlet eigenvalues of the M\"{o}bius strip $M_1$ satisfies,
\begin{equation}\label{counting}
N(\lambda)\geq\frac {\pi} 4\lambda  \,  -2\sqrt{\lambda}+ 1.
\end{equation}
\end{proposition}

\begin{proof}

Let $D_\lambda$ be the part of the closed disk of radius $\sqrt{\lambda}$ in the first quadrant.

In view of Lemma~\ref{L-mps-2} and \eqref{E-mps-6}, to each pair $(m,n)\in D_\lambda$, such that $m+n$ odd and $m\neq 0$, we associate a rectangle or a square, as follows:

\begin{itemize}
\item If $n>0$, we associate the rectangle $ R_{m,n}=[m,m+1]\times[n,n+2]$.
\item If   $n = 0$, (hence $m$ odd positive), we associate the square $R_{m,0}=[m,m+1]\times[1,2]$.
\end{itemize}

\begin{figure}
\centering
\includegraphics[width=0.5\linewidth]{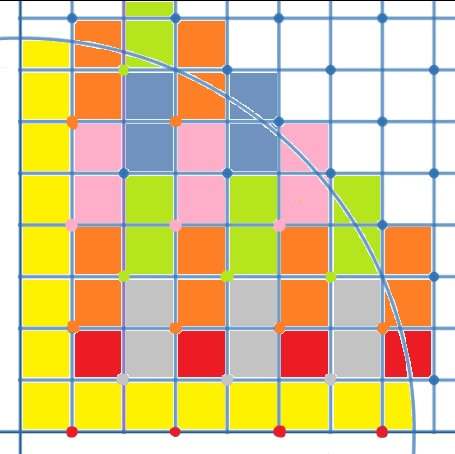}
\caption{The square $R_{m,0}$ (red) and the rectangles $R_{m,n}$ (color depending on $n$)}
\end{figure}%

The following lemma will be useful to continue the proof.

\begin{lemma}
Let $I_\lambda=\{(m,n)\in D_\lambda; m+n \ odd \text{ and } m\neq 0\}$. Then,
\begin{equation}\label{lemmecompte}
\displaystyle \cA(D_\lambda)  \leq \sum_{(m,n)\in I_\lambda} \cA(R_{m,n})  + 2\sqrt{\lambda}-1
\end{equation}

\end{lemma}
\begin{proof}
To prove this lemma, it is enough to prove that,
\begin{equation*}
D_\lambda\subset\ [0,1]\times[0,\sqrt{\lambda}]\  \cup\  [0,\sqrt{\lambda}]\times[0,1] \cup\ \bigcup_{(m,n)\in I_\lambda} R_{m,n}\,.
\end{equation*}

To see that, let us prove that for any number $(x,y)\in D_\lambda$, such that $x>1$ and $y>1$, there is a pair $(m,n)\in I_\lambda$  such that $(x,y) \in R_{m,n}$.
Consider the integer $m$ such that $m\leq x \leq m+1$. Two cases are possible here:

\begin{itemize}
\item If $m$ is odd, let $n$ be the biggest even number less than or equal to $y$: $(m,n)$ is then in $I_\lambda$ and $(x,y)\in R_{m,n}$.

\item If $m$ is even, let $n$ be  the biggest odd number less than or equal to $y$: $(m,n)$ is then in $I_\lambda$ and $(x,y)\in R_{m,n}$.
\end{itemize}

For $x \le 1$ (resp. $y \le 1$), one can easily see that $(x,y)$ belongs to the vertical strip $[0,1]\times[0,\sqrt{\lambda}]$ (resp. the horizontal strip $[0,\sqrt{\lambda}]\times[0,1]$). This ends  the proof of \eqref{lemmecompte}.
\end{proof}

To finish the proof of Proposition \ref{counting}, it is enough to consider the multiplicity of the eigenspace associated to each $\lambdah(m,n)$, one can indeed write,
\begin{equation*}
N(\lambda)=\sum_{(m,n)\in D_\lambda, m\neq 0} \cA( R_{m,n}),
\end{equation*}
which is equivalent to \eqref{counting}.
\end{proof}

\section{Upper bound for the Courant-sharp eigenvalues}\label{S-upper}

This section is inspired by the paper of C. L\'ena on the torus \cite{Le}.

\begin{theorem}\label{127}
For $k\notin \{1,2,7\}$, the eigenvalues $\lambda_k$ of $M_1$ are not Courant-sharp.
\end{theorem}
 Before  starting  the proof of the theorem, we need the following two lemmas.
\begin{lemma} \label{faber-Moebius}
If $\lambda$ is an eigenvalue  of the Laplacian on $M_1$ with an associated eigenfunction $u$, and if  the number of nodal domains $\nu(u)$ of $u$ satisfies $\nu(u) \geq 4$, then we have:

\begin{equation}
\frac{ j_{0,1}^2 \nu(u)}{\pi} \leq \lambda
\end{equation}
\end{lemma}

\begin{proof}
If $\nu(u)\geq 4$, observing that the area of $M=M_1$ is $\pi^2$, there exists one nodal domain $D$ associated to $u$
with area $\cA(D)$ less than $\frac{\pi^2}{4}$, hence less than $\pi$.
Applying now  \eqref{Faber} to $D$, we get:
\begin{equation}
\lambda=\lambda_1(D)\geq \frac {\pi j_{0,1}^2}{\cA(D)}\geq \frac {\nu(u) j_{0,1}^2}{\pi}.
\end{equation}
\end{proof}

We now apply this lemma to a Courant-sharp eigenvalue $\lambda_k$.  By assumption, there exists an eigenfunction $u_k$ with $k$ nodal domains. Applying Lemma \ref{faber-Moebius} to $u_k$, we get:

\begin{proposition}
If $\lambda_k$ is a Courant-sharp eigenvalue of the Laplacian  on $M_1$,  with $k\geq 4$, then,
\begin{equation}\label{sharp-faber}
\frac{\lambda_k}{k}\geq \frac{(j_{0,1})^2}{\pi}\,.
\end{equation}
\end{proposition}

We can now give the proof of the theorem.

 \begin{proof}[Proof of Theorem \ref{127}]

Assume that $\lambda_k$ is Courant-sharp. By Remark~\ref{R-int-2}, $N(\lambda_k) \le k-1$. Using \eqref{counting} and \eqref{sharp-faber}, we conclude that $P(\sqrt{\lambda _k}) \ge 0$, where
 \begin{equation}\label{sharp-quadratic}
P(x):= (\frac\pi {j_{0,1}^2} -\frac\pi 4) x^2 +2 x -2\,.
 \end{equation}

\begin{table}[!htb]
\centering
\resizebox{0.5\textwidth}{!}{%
\begin{tabular}[c]{|c|c|c|c|}%
\hline&&&\\[2pt]
Eigenvalue & $(m,n)$ & $\lambda_k$ & $\frac{\lambda_k\, \pi}{(j_{0,1})^2}$\\[7pt]
\hline
$\lambda_1$ & $(1,0)$ & $1$ & -- \\[5pt] 
\hline
$\lambda_2$ & $(1,2), (2,1)$ & $5$ & --\\[5pt] 
\hline
$\lambda_6$ & $(3,0)$ & $9$  & $4.8891$\\[5pt]
\hline
$\lambda_7$ & $(2,3), (3,2)$ & $13$ & $7.0620$\\[5pt]
 \hline
$\lambda_{11}$ & $(1,4), (4,1)$ & $17$ & $9.2349$\\[5pt]
\hline
$\lambda_{15}$ & $(3,4), (4,3), (5,0)$ & $25$ & $13.5807$\\[5pt]
\hline
$\lambda_{20}$ & $(2,5), (5,2)$ & $29$  & $15.7536$\\[5pt]
\hline
$\lambda_{24}$ & $(1,6), (6,1)$ & $37$  & $20.0995$\\[5pt]
\hline
$\lambda_{28}$ & $(4,5), (5,4)$ & $41$ & $22.2724$\\[5pt]
\hline
$\lambda_{32}$ & $(3,6), (6,3)$ & $45$ & $24.4453$\\[5pt]
\hline
$\lambda_{36}$ & $(7,0)$ & $49$ & $26.6182$\\[5pt]
\hline
$\lambda_{37}$ & $(2,7), (7,2)$ & $53$ & $28.7911$\\[5pt]
\hline
$\lambda_{41}$ & $(5,6), (6,5)$ & $61$ & $33.1370$\\[5pt]
\hline
$\lambda_{45}$ & $(1,8), (8,1), (4,7), (7,4)$ & $65$ & $35.3099$\\[5pt]
\hline
\end{tabular}}
\medskip
\caption{Courant-sharp vs Faber-Krahn}\label{T-CSFK}
\end{table}

Since the quadratic function $P$ is negative for $x \ge 8$, Courant-sharp Dirichlet eigenvalues of $M_1$ must be smaller that $64$. The last step is to check the eigenvalues smaller than $64$. Table~\ref{T-CSFK} lists the eigenvalues $\lambda_k$ less than or equal to $65$ which could be Courant-sharp, taking Remark~\ref{R-int-2} into account, together with the corresponding ratio $\frac{\lambda_k\, \pi}{(j_{0,1})^2}$ (which only makes sense for $k\ge 4$). The only eigenvalue which satisfies the Faber-Krahn condition $k \le \frac{\lambda_k\, \pi}{(j_{0,1})^2}$ is $\lambda_7$.\medskip

According to Lemma~\ref{L-es23-2}, the eigenvalue $\lambda_7$  is not Courant-sharp.

\end{proof}

Taking Proposition~\ref{P-esmn-4} into account, the proof of Theorem~\ref{main} is complete.

\FloatBarrier

\section{ An Euler-type formula for the M\"{o}bius band}\label{S-eulerm}

The M\"{o}bius band is a non-orientable surface with boundary, with Euler characteristic $0$, and one boundary component. Otherwise stated, it is a real projective plane with one disk removed. A natural question is whether the non-orientability character can be detected in the nodal patterns. The purpose of this section is to give a positive answer.\medskip

Let $\Omega$ be a bounded connected open set in $\mathbb R^2$, with piecewise $C^{1,+}$ boundary. The following theorem appears in \cite{HOMiNa}.

\begin{proposition}\label{P-euler-2}
Let $\mathcal D = \set{D_i}_{i=1}^k$ be a nodal $k$-partition of $\Omega$, where the $D_i$'s are the $k$ nodal domains of some Dirichlet eigenfunction $\Phi$ in $\Omega$, and $\partial \cD = \cZ(\Phi)$. Let $b_0$ denote the number of connected components of $\partial \Omega$, and $b_1$ the number of connected components of $\partial \mathcal D \cup\partial \Omega$. Let $\cC_i(\partial \cD)$ denote the set of interior critical zeros of $\Phi$, and $\cC_b(\partial \cD)$ the set of boundary critical zeros. Given ${\bf x} \in \cC_i(\partial \cD)$, let $\nu({\bf x})$ denote the number of nodal semi-arcs at ${\bf x}$; given ${\bf y} \in \cC_b(\partial \mathcal D)$, let $\rho({\bf y})$ denote the number of nodal semi-arcs hitting the boundary at ${\bf y}$. Then,
\begin{equation}\label{E-euler-2}
k= 1 + b_1-b_0+ \frac{1}{2}\, \sum_{{\bf x}\in \cC_i(\partial \mathcal D )}
\left( \nu({\bf x}) - 2\right) + \frac{1}{2} \, \sum_{{\bf y}\in \cC_b(\partial \mathcal D )}\rho({\bf y})\,.
\end{equation}
\end{proposition}

In the case of the M\"{o}bius strip, we expect the following Euler-type formula to hold. It takes non-orientability into account.

\begin{property}\label{P-eulerm-2}
Let $\cD$ be nodal $k$-partition of the M\"{o}bius strip $M_1$. With the previous notation, the following relation holds,
\begin{equation}\label{E-eulerm-2}
k = \omega(\cD)  + b_1-b_0 + \frac{1}{2}\, \sum_{{\bf x}\in \cC_i(\partial \mathcal D )} \left(\nu({\bf x}) - 2 \right) + \frac{1}{2} \, \sum_{{\bf y}\in \cC_b(\partial \mathcal D )}\rho({\bf y})\,,
\end{equation}
where $\omega(\cD) = 0$ if all the nodal domains are orientable, and $\omega(\cD) = 1$ if one nodal domain in non-orientable. In the latter case, it turns out that there is exactly one non-orientable nodal domain.
\end{property}

\pf We refer to \cite{BH-em} for the proof of Property~\ref{P-eulerm-2} in the more general framework of partitions. \hfill \qed

\begin{remark}\label{R-eulerm-2} The Euler-type formula \eqref{E-euler-2} actually holds for more general partitions, \cite{BH0, HOMiNa}.
\end{remark}

One can easily check that Property~\ref{P-eulerm-2} is true for the nodal domains which appear in Section~\ref{S-es23}, with $\omega(\cD) = 0$ for all nodal patterns, except for the nodal pattern in Figure~\ref{F-es23-gc}\,(C) for which $\omega(\cD)=1$, note that the nodal domain labeled ``2'' is homeomorphic to a M\"{o}bius strip. \medskip

Formula~\eqref{E-eulerm-2} can also be verified on the nodal patterns displayed in the following figures. Figure~\ref{F-fnp-sin} displays nodal patterns containing one M\"{o}bius strip, and one or two non-simply connected, orientable nodal domains. Figure~\ref{F-fnp-moeb} displays the nodal patterns of the function
$$
\cos\theta \sin(x)\cos(my) + \sin\theta \sin(mx)\cos(y)\,, \text{~with~} m=6,
$$
for the values $\theta=0.18\,\pi$ and $\theta=0.4\,\pi$. In both cases, one of the nodal domains is non-orientable, homeomorphic to a M\"{o}bius strip (left), or to a M\"{o}bius with holes (right).\medskip

\begin{figure}[h]
\centering
\begin{subfigure}[t]{.45\textwidth}
\centering
\includegraphics[width=\linewidth]{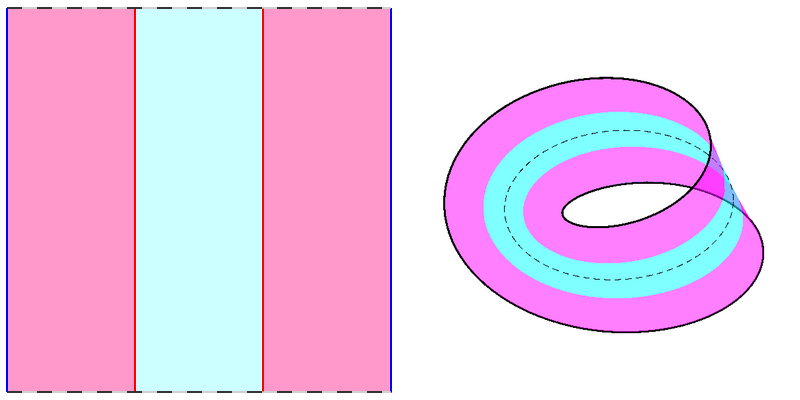}
\end{subfigure}
\hspace{3mm}
\begin{subfigure}[t]{.45\textwidth}
\centering
\includegraphics[width=\linewidth]{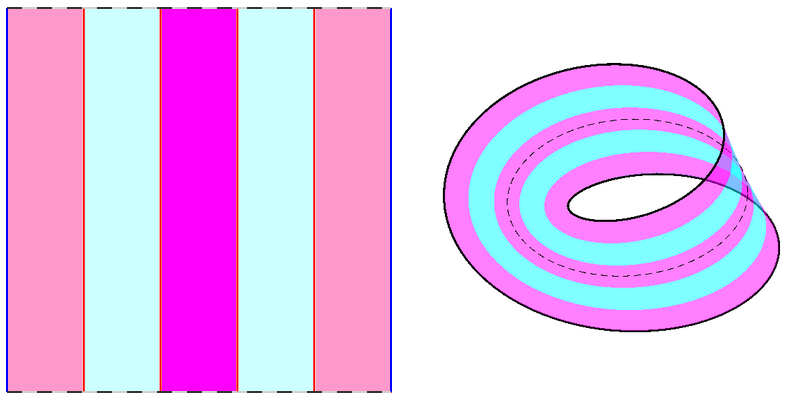}
\end{subfigure}
\caption{Nodal sets of $\sin(3x)$ and $\sin(5x)$}\label{F-fnp-sin}
\end{figure}

\begin{figure}[h]
\centering
\begin{subfigure}[t]{.35\textwidth}
\centering
\includegraphics[width=\linewidth]{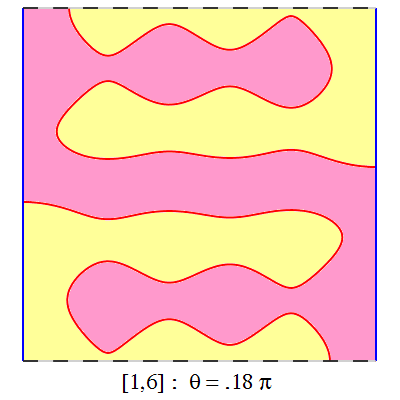}
\end{subfigure}
\begin{subfigure}[t]{.35\textwidth}
\centering
\includegraphics[width=\linewidth]{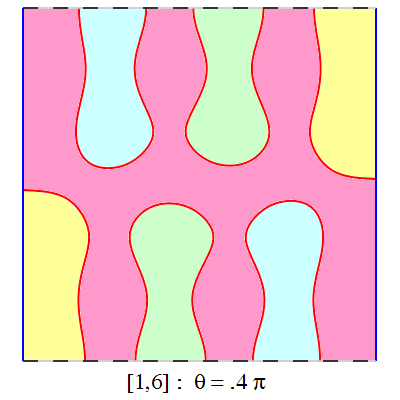}
\end{subfigure}
\caption{Nodal sets with non-orientable nodal domains}\label{F-fnp-moeb}
\end{figure}

\FloatBarrier

\section{High energy eigenfunctions with two nodal domains}\label{S-stern}

From \eqref{E-mps-6}, we conclude that some of the Dirichlet eigenfunctions of the square $(0,\pi)^2$ are also Dirichlet eigenfunctions of the M\"{o}bius strip. This is in particular the case of the eigenfunctions $\sin(x)\sin(2ry)$ and $\sin(2rx)\sin(y)$, where $r$ is a positive integer. According to a result of A.~Stern, for $\varepsilon > 0$ small enough (depending on $r$), the eigenfunction $\sin(x)\sin(2ry) + (1+\varepsilon)\, \sin(2rx)\sin(y)$ has precisely two nodal domains. This is illustrated in Figure~\ref{F-stern-14} and \ref{F-stern-16}, we refer to \cite{BH} for detailed proofs. Other examples involve linear combinations of $\sin(x)\cos(2ry)$ and $\sin(2ry)\cos(y)$, see for example Figure~\ref{F-fnp-moeb} (left).\medskip

\begin{figure}[!hbp]
\centering
\includegraphics[width=0.8\linewidth]{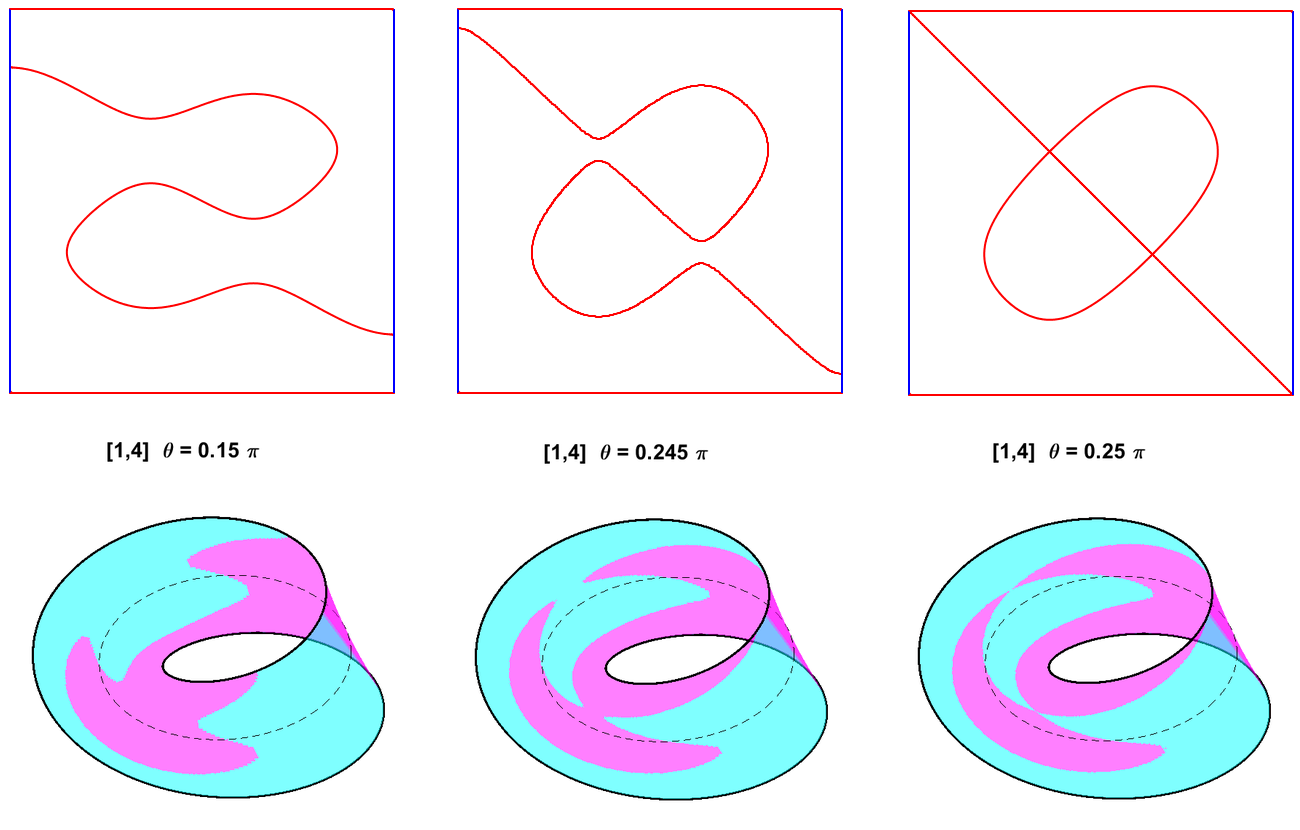}\\
\caption{Example \`{a} la Stern $[1,4]$}\label{F-stern-14}
\end{figure}

\begin{figure}[ht]
  \centering
  \includegraphics[width=0.8\linewidth]{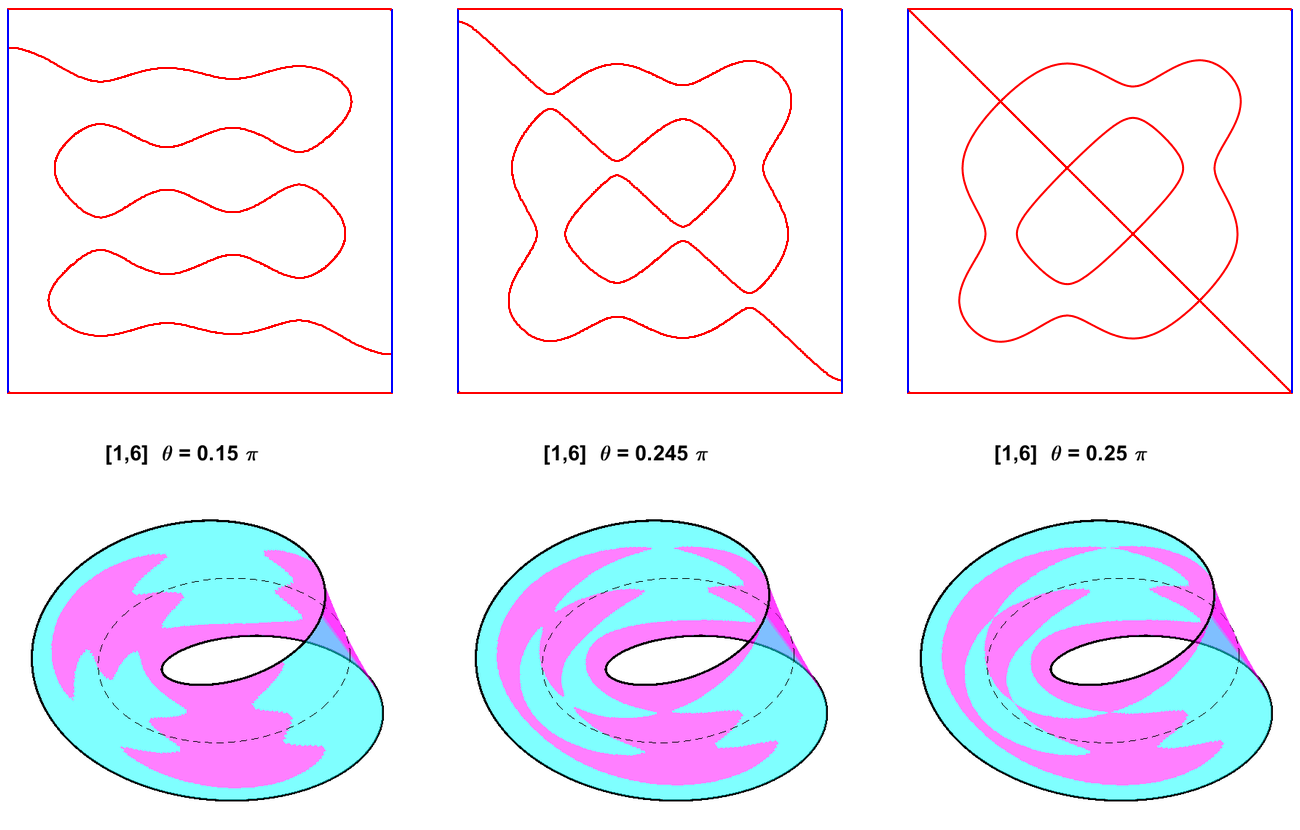}\\
  \caption{Example \`{a} la Stern $[1,6]$}\label{F-stern-16}
\end{figure}

\vspace{1cm}
\bibliographystyle{plain}

\end{document}